\documentclass[12pt,reqno]{article}

\usepackage[usenames]{color}
\usepackage{amssymb}
\usepackage{graphicx}
\usepackage{amscd}

\usepackage[colorlinks=true,
linkcolor=webgreen,
filecolor=webbrown,
citecolor=webgreen]{hyperref}

\definecolor{webgreen}{rgb}{0,.5,0}
\definecolor{webbrown}{rgb}{.6,0,0}

\usepackage{xcolor}
\usepackage{fullpage}
\usepackage{float}

\usepackage[american]{babel}
\usepackage{array}
\usepackage{amsmath}
\usepackage{amsthm}
\usepackage{amsfonts}
\usepackage{latexsym}
\usepackage{epsf}
\usepackage{wasysym}
\usepackage{skak}
\usepackage{txfonts}
\usepackage{numprint}

\usepackage{multirow}
\usepackage{tabu}

\usepackage{algorithm} 
\usepackage{algpseudocode}

\usepackage{mathtools}

\DeclareMathOperator{\Tr}{Tr}

\newcommand{\seqnum}[1]{\href{https://oeis.org/#1}{\rm \underline{#1}}}

\newcommand{\uncovered}{\ensuremath{\medcirc}}
\newcommand{\occupied}{\ensuremath{\color{black}\medbullet}}
\newcommand{\covered}{\ensuremath{\color{gray}\medbullet}}
\newcommand{\unc}{\uncovered}
\newcommand{\occ}{\occupied}

\newcommand*{\defeq}{\mathrel{\vcenter{\baselineskip0.5ex \lineskiplimit0pt
                     \hbox{\scriptsize.}\hbox{\scriptsize.}}}
                     =}

\begin{document}

\begin{center}
  \includegraphics[width=1in]{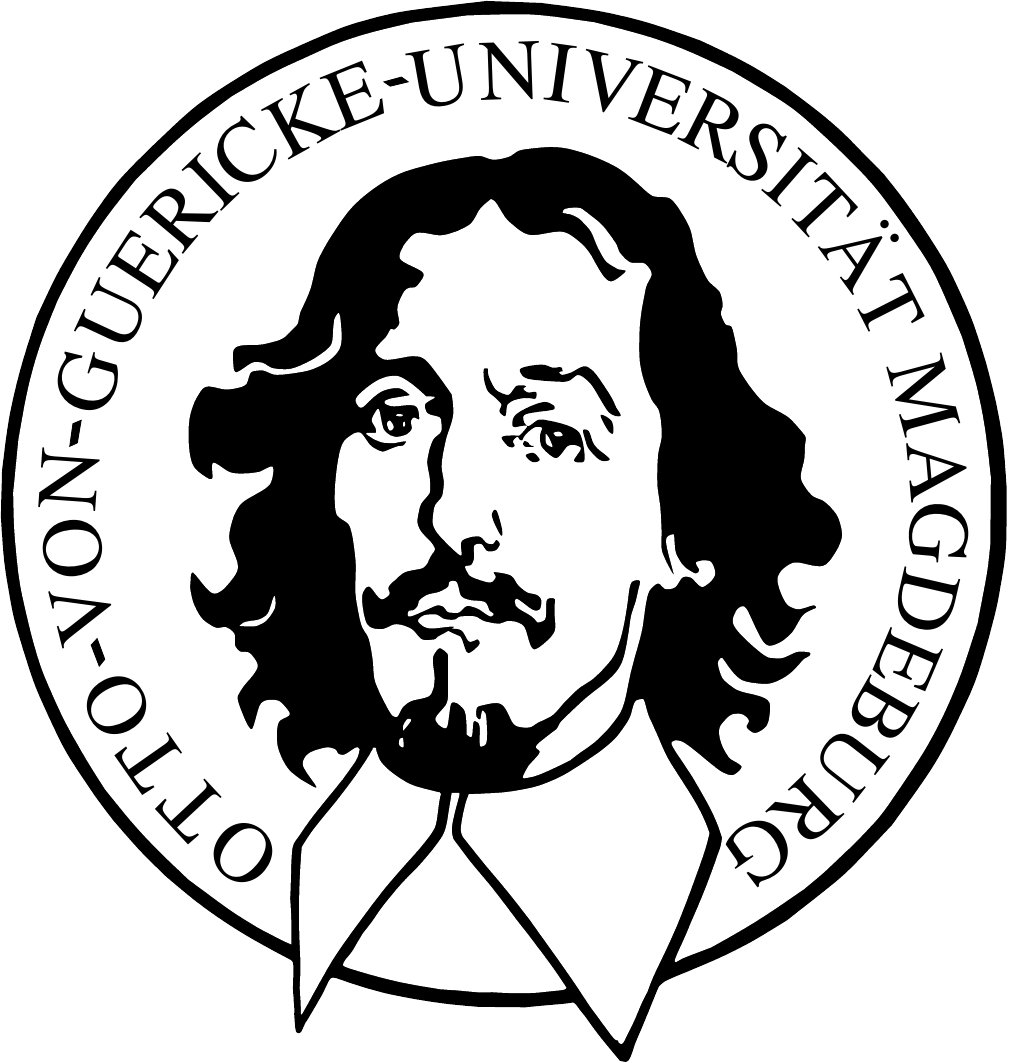} \hfill
  \includegraphics[width=1in]{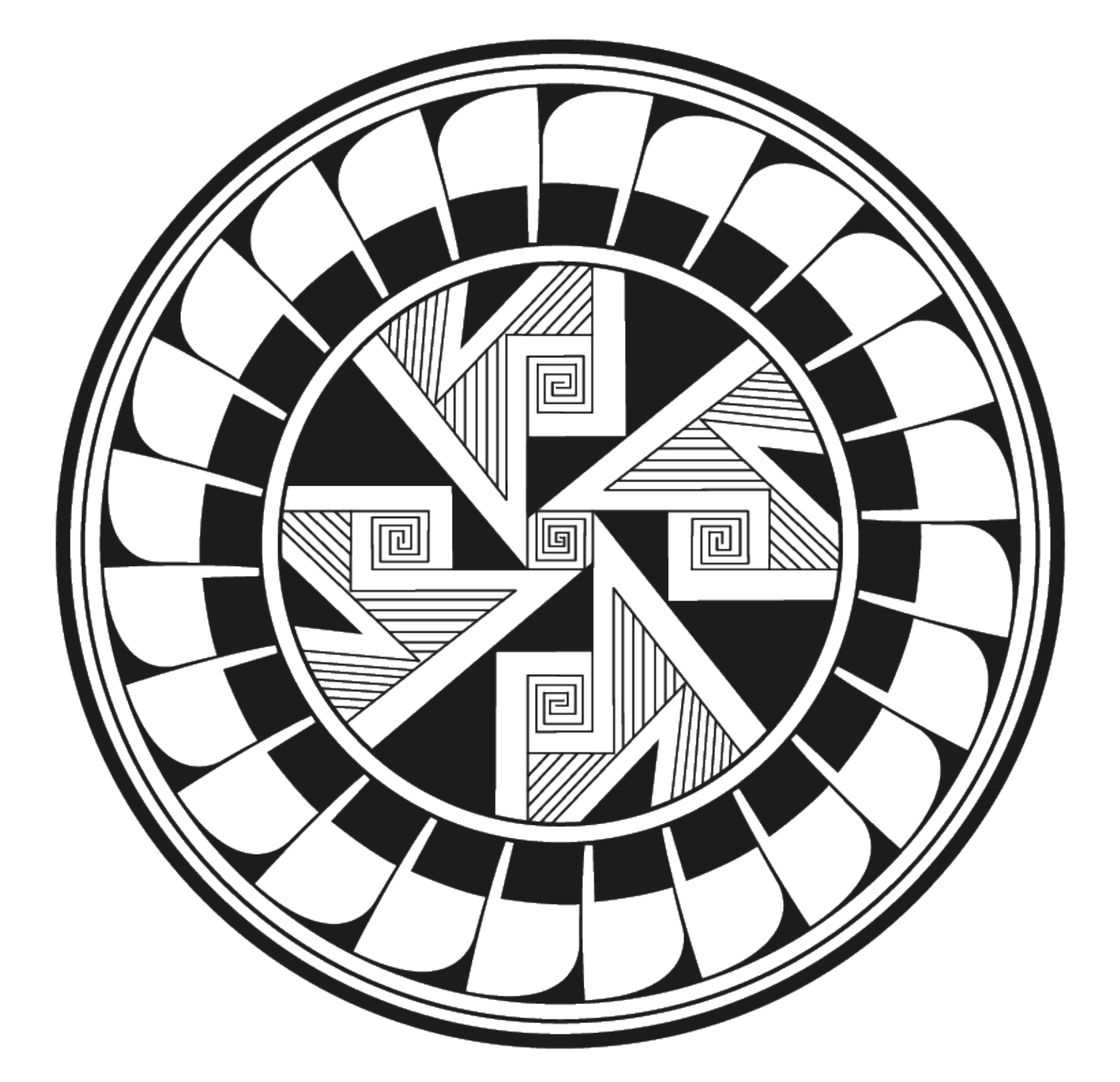}
\end{center}

\theoremstyle{plain}
\newtheorem{theorem}{Theorem}
\newtheorem{corollary}[theorem]{Corollary}
\newtheorem{lemma}[theorem]{Lemma}
\newtheorem{proposition}[theorem]{Proposition}

\theoremstyle{definition}
\newtheorem{definition}[theorem]{Definition}
\newtheorem{example}[theorem]{Example}
\newtheorem{conjecture}[theorem]{Conjecture}
\newtheorem{Observation}[theorem]{Observation}

\theoremstyle{remark}
\newtheorem{remark}[theorem]{Remark}

\begin{center}
\vskip 1cm{\LARGE\bf Domination Polynomials of 
  the Grid, the Cylinder, the Torus, and the King Graph
}
\vskip 1cm
\large
Stephan Mertens \\
{Institut f\"ur Physik \\
Otto-von-Guericke Universit\"at Magdeburg \\
Postfach 4120\\
39016 Magdeburg\\
Germany} \\
and \\
Santa Fe Institute\\
1399 Hyde Park Rd\\
Santa Fe, NM 87501\\
USA\\
\href{mailto:mertens@ovgu.de}{\texttt{mertens@ovgu.de}}
\end{center}

\vskip .2 in

\begin{abstract}
  We present an algorithm to compute the domination polynomial of the
  $m \times n$ grid, cylinder, and torus graphs and the king graph. The time
  complexity of the algorithm is $O(m^2n^2 \lambda^{2m})$ for the torus and 
  $O(m^3n^2\lambda^m)$ for the other graphs, where $\lambda = 1+\sqrt{2}$. 
  The space complexity is $O(mn\lambda^m)$ for all of these graphs. We use this 
  algorithm to compute domination polynomials for graphs up to
  size $24\times 24$ and the total number of dominating sets for even
  larger graphs. This allows us to give precise estimates of the
  asymptotic growth rates of the number of dominating sets. We
  also extend several sequences in the Online Encyclopedia of Integer
  Sequences. 
\end{abstract}

\section{Introduction}

A dominating set in a graph $G = (V , E)$ is a subset $S \subseteq V$
of vertices such that every node in $V$ is either an element of $S$ or
has a neighbor in $S$.

Domination is one of the most widely studied topics in graph
theory. According to Haynes, Hedetniemi, and Henning \cite{haynes:hedetniemi:henning:20}, more than
4000 papers on the subject were published by the year 2020.
Domination problems originated in the 19th century in chess
\cite{jaenisch:62, hedetniemi:hedetniemi:21}. A placement of chess pieces on a chessboard
is called dominating if each free square of the chessboard is under
attack by at least one piece. Figure~\ref{fig:kings-graph} shows 
$9$ kings dominating the $8\times 8$ chessboard.

\begin{figure}
\begin{tabular}{p{0.5\textwidth}p{0.5\textwidth}}  
  \begin{center}
    \fenboard{8/1k2k2k/8/8/1k2k2k/8/8/1k2k2k - - - 0 0}
    $$\showboard$$
  \end{center} &
   \vspace{4mm}
   \begin{center}
     \includegraphics[width=0.353\textwidth]{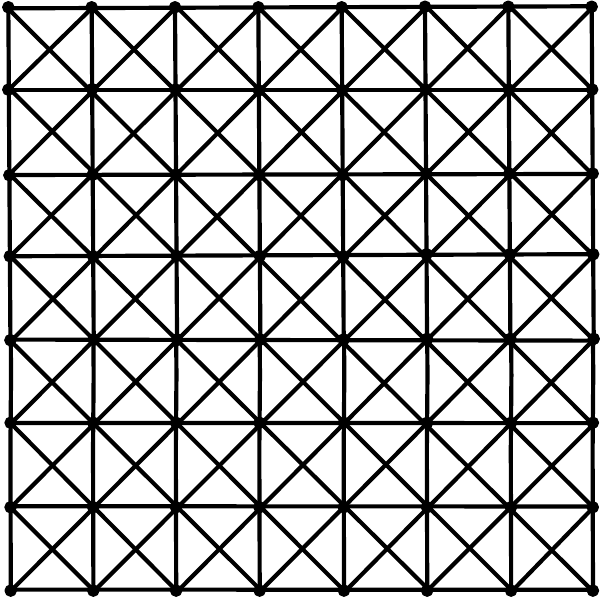}
  \end{center}\\[-6ex]
\end{tabular}
\caption{Nine kings are required to dominate the
  chessboard, or equivalently the $8\times 8$ king graph. }
\label{fig:kings-graph}
\end{figure}

The link from chess to graph theory is given by graphs like the king
graph (Figure~\ref{fig:kings-graph}).  In this graph, the vertices
represent the squares of the board, and each edge represents a legal
move of a king. Obviously a dominating placement of kings on the board
corresponds to a dominating set of the king graph. Graphs for 
other chess pieces can be defined analogously. 

\begin{figure}
  \centering
  \includegraphics[width=0.9\linewidth]{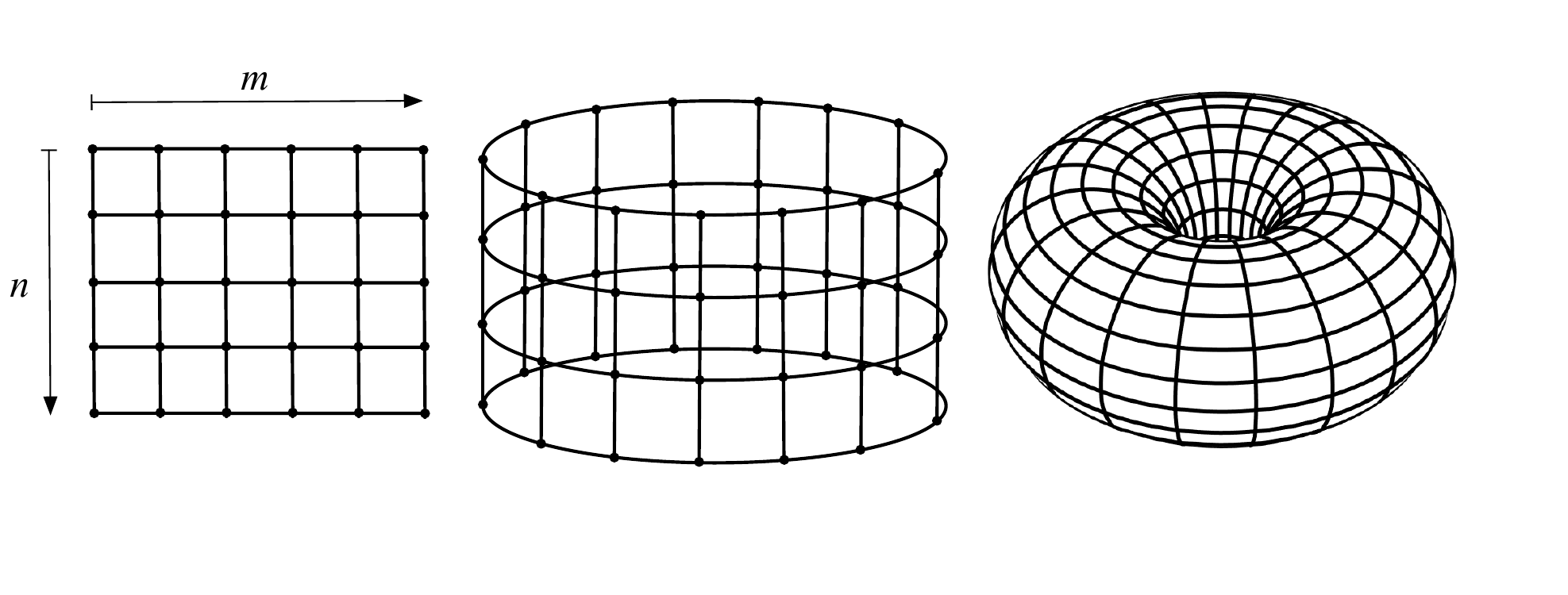}
  \caption{Examples of an $m\times n$ grid, cylinder
    and torus graph.}
  \label{fig:grid-cylinder-graph}
\end{figure}

In this contribution, we  study domination in some related families of graphs: 
the $m\times n$ grid, cylinder, and torus graph
(Figure~\ref{fig:grid-cylinder-graph}). If $P_n$ denotes the path graph
of $n$ vertices, the grid graph $G_{m\times n}$ is the Cartesian graph product
$G_{m\times n} = P_m \square P_n$. The cylinder graph is
$G_{\overline{m}\times n} = C_m \square P_n$, where $C_m$
is the cycle graph: we use the overbar $\overline{m}$ 
to indicate the cyclic index. The $m\times n$ torus graph is
$G_{\overline{m}\times\overline{n}} =C_m \square C_n$.
We also study domination in the $m\times n$ king graph
$K_{m\times n} = P_m \boxtimes P_n$, the strong graph product
of two path graphs $P_m$ and $P_n$. 

Our goal is to compute the \emph{domination polynomial} of these graphs. 
The domination polynomial of a graph $G$ is the
generating function of its dominating sets with respect to their size, i.e.,
\begin{equation}
  \label{eq:def_domination_polynomial}
  D_G(z) = \sum_{S\subseteq V}  z^{|S|}\,,
\end{equation}
where the sum runs over all dominating sets in $G$. 
Like other graph polynomials, the domination polynomial encodes many
interesting properties of a graph
\cite{akbari:alikhani:peng:10}.  For example, the lowest power of $z$ appearing in 
$D_G(z)$ is the size of the smallest dominating set in $G$, which is 
known as the \emph{domination number} $\gamma(G)$. 
A considerable fraction of the 4000 papers 
mentioned above are studies of $\gamma(G)$ for various graphs. 

A closed form for the domination polynomial is known only for a few families of simple graphs like 
complete graphs, path and cycle graphs, wheel graphs, and star graphs 
\cite{alikhani:peng:09, alikhani:peng:08,alikhani:peng:14}. Recently, rook graphs on $m \times n$ chessboards, 
i.e., Cartesian products of two complete graphs $K_m \square K_n$,
were added to this list \cite{mertens:24a}. As far as we
know, the rook is the first chess piece for which this has been achieved.

Thus for most graphs, the domination polynomial can only be computed numerically. 
However, this approach is challenging because of its computational
complexity.  In theoretical computer science, the decision problem \textsc{Dominating Set} 
asks, given a graph $G$ and an integer $k$, whether $G$ has a dominating set of size at most
$k$, i.e., whether $\gamma(G) \le k$. This problem is NP-complete, which 
can be shown by reduction from the \textsc{Vertex Cover} problem \cite{garey:johnson:79}. 
Unless $\mathrm{P}=\mathrm{NP}$, this means that no polynomial-time algorithm exists 
to compute $\gamma(G)$. Indeed, the fastest known algorithm to find the dominating set of minimum
size for general graphs $G=(V,E)$ has time complexity
$O(1.4969^{|V|})$ \cite{van-rooij:bodlaender:11}. 

The grid graph is planar and bipartite, and both of these properties often allow
polynomial time algorithms for problems that are NP-complete for
general graphs \cite{NOC}. But not in this case: \textsc{Dominating
  Set} is NP-complete even on subgraphs of the $m\times n$ grid
\cite{clark:colbourn:johnson:90}.  
This suggests that we should not expect to find a polynomial-time algorithm that computes $\gamma(G)$ 
for the grid graph or its relatives, let alone one that computes the entire polynomial $D_G(z)$. 
But we can try to reduce the exponential running time as much as possible. 
This is the main goal of this work. 

The paper is organized as follows. We begin by proving in Section \ref{sec:transfermatrix} that
the domination polynomial of the $m\times n$ grid, cylinder, and torus 
can be expressed in terms of the $n$th power of an $a \times a$ matrix $A$, the
``transfer matrix'', where $a=O(\lambda^m)$ with
$\lambda = 1+\sqrt{2}=2.4142\ldots$\  As we show in Section \ref{sec:algorithm}, 
this gives rise to an algorithm for the domination polynomial with time complexity
$O(m^3n^2\lambda^m)$ for the grid and cylinder, and a closely related algorithm with time complexity
$O(m^3n^2\lambda^{2m})$ for the torus. 
We also explain how this algorithm can be adapted to compute the 
domination polynomial of the king graph, with the same time and space
complexity.
Since the running times of these algorithms are exponential in the width $m$ as opposed to the 
number of vertices $|V|=mn$, they represent a considerable improvement over the algorithm of \cite{van-rooij:bodlaender:11}.

 In Sections \ref{sec:results} and \ref{sec:growthrate} we discuss some numerical results obtained by our algorithm. 
 In particular, we estimate the asymptotic growth rate of the total
 number of dominating sets for all these graphs. We give 
 Conclusions in Section \ref{sec:conclusions} and provide some combinatorial proofs and domination polynomials in 
 the appendices.

\section{The transfer matrix}
\label{sec:transfermatrix}

The idea of the transfer matrix approach is to compute the domination
polynomial of a grid or cylinder row by row. We do this by defining legal transitions 
from one row to the next, identifying which vertices in each row are in the dominating set $S$. 
Consider the following definitions, where we borrow some wording from domination problems in chess.

\begin{definition}
  Given a graph $G=(V,E)$ and a set $S \subseteq V$, we say a vertex $v \in V$ is
   \emph{occupied} if $v \in S$, \emph{covered} if $v \notin S$ but some neighbor of $v$ is in $S$, 
   and \emph{uncovered} if $v \notin S$ and no neighbor of $v$ is in $S$. 
\end{definition}

Clearly every vertex is either occupied, covered, or uncovered. As we construct $S$, some vertices 
in the current row may be uncovered, because they will become covered by a neighboring occupied vertex in the next row.
This gives us the following definition.

\begin{definition}
  \label{def:almost_dominating}
 If $G$ is the grid graph $G_{m\times n}$ or the cylinder graph $G_{\overline{m}\times n}$, 
   we say $S$ is \emph{almost dominating} if every vertex 
  in the subgraph $G_{m\times (n-1)}$ (resp., $G_{\overline{m}\times (n-1)}$) 
 consisting of the first $n-1$ rows are occupied or covered.
\end{definition}

We label vertices according to their state, namely \occupied\ (occupied), \covered\ (covered), and \uncovered\ (uncovered). 
Given an almost dominating set $S$, we define its \emph{signature} $\sigma$ as the string of length $m$ over the alphabet $\{\uncovered,\covered,\occupied\}$ that identifies the states of the vertices in the $n$th row. However, not all such strings can occur: since the neighbors of an occupied vertex are covered, the symbols \occupied\ and \uncovered\ cannot be adjacent. 
\begin{definition}
  \label{def:signature}
  A \emph{signature} of length $m$ is a string $\sigma$ of
  length $m$ over the alphabet $\{\uncovered,\covered,\occupied\}$ which does not contain either of
  the substrings (\uncovered,\occupied) and (\occupied,\uncovered). A \emph{cyclic signature} is one where 
  this substring constraint also applies to the pair $(\sigma_1,\sigma_m)$.
\end{definition}
\noindent
Signatures apply to the grid, and cyclic signatures apply to the cylinder.

The time and space complexity of our algorithms depend on the number of signatures or cyclic signatures. 
The number $a(m)$ of signatures is given by
\begin{align}
    \label{eq:a_signatures}
    a(m) &= \frac{1}{2}\left(1-\sqrt{2}\right)^{m+1} +
           \frac{1}{2}\left(1+\sqrt{2}\right)^{m+1}\\
           &= 3, 7, 17, 41, 99, 239, 577, 1393, 3363, 8119, \ldots \nonumber
\end{align}
This sequence has two entries in the OEIS, \seqnum{A001333} and
\seqnum{A078057}, differing only in their offset. Since cyclic signatures are constrained at one more pair, the total number $\overline{a}(m)$ of cyclic signatures is smaller than $a(m)$ for $m \ge 3$, although with the same asymptotic growth rate:
\begin{align}
  \label{eq:c_signatures}
  \overline{a}(m) &= 1 + \left(1-\sqrt{2}\right)^m + \left(1+\sqrt{2}\right)^m \\
         &= 3, 7, 15, 35, 83, 199, 479, 1155, 2787, 6727, \ldots \nonumber
\end{align}
This is \seqnum{A124696} in the OEIS. We derive the formulas for
$a(m)$ and $\overline{a}(m)$ in the Appendix.

Slightly abusing notation, we write $G_{m\times n}(z)$ for the
domination polynomial of the grid $G_{m\times n}$ and
$G_{\overline{m}\times n}(z)$ for the domination polynomial of the
cylinder $G_{\overline{m}\times n}$. We also write 
$G_{m\times n}^{\sigma}(z)$  and $G_{\overline{m}\times n}^{\sigma}(z)$ for 
the generating functions of almost dominating sets on $G_{m\times n}$
(resp., $G_{\overline{m}\times n}$) with signature $\sigma$. 
The connection between these dominating polynomials and almost-dominating polynomials is then given by 
the following lemma.
\begin{lemma}
  \label{lem:P_sigma}
  Let $\unc(\sigma)$ denote the number of uncovered
  vertices in $\sigma$. Then
  \begin{equation}
    \label{eq:almost-full}
    G_{m\times n}(z) = \sum_{\sigma\, :\, \unc(\sigma)=0} G_{m\times
                       n}^\sigma(z)\,,\qquad
    G_{\overline{m}\times n}(z) = \sum_{\sigma\, :\, \unc(\sigma)=0} G_{\overline{m}\times
                       n}^\sigma(z)\,,
  \end{equation}
  where the sum for the grid (resp., the cylinder) runs over all signatures (resp., cyclic signatures). 
\end{lemma}

\begin{proof}
  The dominating sets of $G_{m\times n}$ and $G_{\overline{m}\times n}$ consist of the 
  almost dominating sets which are in fact dominating, i.e., where there are no uncovered vertices in the $n$th row.
\end{proof}

Now we have all the ingredients to implement the idea of constructing dominating 
sets row by row. Consider an almost dominating set in an $m\times n$ grid or cylinder with signature
$\sigma$, and consider adding an $(n+1)$st row with signature $\tau$. Only certain pairs $\sigma, \tau$ 
are compatible. Wherever $\sigma$ has an uncovered vertex, its neighbor in $\tau$ must be occupied. 
Similarly, wherever $\sigma$ is occupied, its neighbor in $\tau$ is occupied or covered by definition. 
Finally, a vertex in $\tau$ cannot be covered unless it has an occupied neighbor, either above it in $\sigma$ 
or to either side in $\tau$. Thus the new signature $\tau$ must be compatible with the
previous signature $\sigma$ according to the following definition.
\begin{definition}
\label{def:compatible}
  A (cyclic) signature $\tau=(\tau_1,\ldots,\tau_m)$ is 
  \emph{compatible} with a (cyclic) signature
  $\sigma=(\sigma_1,\ldots,\sigma_m)$ if, for all $i=1,\ldots,m$,
  \begin{equation}
  \begin{aligned}
    \label{eq:def-compatible}
    \sigma_i=\uncovered &\,\,\Longrightarrow\,\, \tau = \occupied \,,\\
    \sigma_i=\occupied &\,\,\Longrightarrow\,\, \tau_i \in \{\covered, \occupied\} \,,\\
    \tau_i=\covered &\,\,\Longrightarrow\,\,(\sigma_i = \occupied) 
    \;\text{or}\; (\tau_{i-1}=\occupied) \;\text{or}\; (\tau_{i+1} = \occupied) \,.
  \end{aligned}
\end{equation}
In the last equation, we compute the indices $i \pm 1 \bmod m$ for cyclic signatures, and ignore $\tau_0$ and $\tau_{m+1}$ in the non-cyclic case.
\end{definition}

Finally, we define the transfer matrices $A$ and $\overline{A}$,
whose rows and columns are indexed by (cyclic) signatures.
\begin{definition}
  Let $\occ(\sigma)$ denote the number of occupied vertices in a (cyclic) signature $\sigma$. 
  For a given integer $m \ge 0$, the transfer matrix $A = (A_{\tau,\sigma})$ 
  is defined as
  \begin{equation}
    \label{eq:transfer-matrix}
    A_{\tau,\sigma} = \begin{cases}
      z^{\occ(\tau)} & \text{if $\tau$ is compatible with $\sigma$} \\
      0 & \text{otherwise,}
    \end{cases}
  \end{equation}
  where $\tau$ and $\sigma$ range over all signatures of length $m$. 
  The transfer matrix $\overline{A} = (\overline{A}_{\tau,\sigma})$ is defined similarly with $\tau$ and $\sigma$ ranging over cyclic signatures of length $m$.
  \label{def:transfer-matrix}
\end{definition}

\begin{table}
  \centering
  {
 \begin{tabular}{c|ccccccc}
 & \uncovered \uncovered  & \uncovered \covered  & \occupied\occupied & \occupied\covered  & \covered \uncovered  & \covered \occupied & \covered \covered \\\hline
\uncovered \uncovered  & 0   & 0   & 0   & 0   & 0   & 0   &  1 \\
\uncovered \covered  & 0   & 0   & 0   & 0   & 0   &  1  & 0  \\
\occupied\occupied &  $z^2$ &  $z^2$ &  $z^2$ &  $z^2$ &  $z^2$ &  $z^2$ &  $z^2$\\
\occupied\covered  & 0   &  $z$  &  $z$  &  $z$  & 0   &  $z$  &  $z$ \\
\covered \uncovered  & 0   & 0   & 0   &  1  & 0   & 0   & 0  \\
\covered \occupied & 0   & 0   &  $z$  &  $z$  &  $z$  &  $z$  &  $z$ \\
\covered \covered  & 0   & 0   &  1  & 0   & 0   & 0   & 0  \\
 \end{tabular}
}
\caption{The transfer matrix $A$ for domination on grids of width $m=2$. Rows and columns are indexed by the new and old signatures $\tau$ and $\sigma$ respectively.}
\label{tab:transfermatrix}
\end{table}

\begin{table}
  \centering
  {
    \setlength{\tabcolsep}{2pt} 
 \begin{tabular}{c|ccccccccccccccc}
 & \uncovered \uncovered \uncovered  & \uncovered \uncovered \covered  & \uncovered \covered \uncovered  & \uncovered \covered \covered  & \occupied\occupied\occupied & \occupied\occupied\covered  & \occupied\covered \occupied & \occupied\covered \covered  & \covered \uncovered \uncovered  & \covered \uncovered \covered  & \covered \occupied\occupied & \covered \occupied\covered  & \covered \covered \uncovered  & \covered \covered \occupied & \covered \covered \covered \\\hline
\uncovered \uncovered \uncovered  & 0   & 0   & 0   & 0   & 0   & 0   & 0   & 0   & 0   & 0   & 0   & 0   & 0   & 0   &  1 \\
\uncovered \uncovered \covered  & 0   & 0   & 0   & 0   & 0   & 0   & 0   & 0   & 0   & 0   & 0   & 0   & 0   &  1  & 0  \\
\uncovered \covered \uncovered  & 0   & 0   & 0   & 0   & 0   & 0   & 0   & 0   & 0   & 0   & 0   &  1  & 0   & 0   & 0  \\
\uncovered \covered \covered  & 0   & 0   & 0   & 0   & 0   & 0   & 0   & 0   & 0   & 0   &  1  & 0   & 0   & 0   & 0  \\
\occupied\occupied\occupied &  $z^3$ &  $z^3$ &  $z^3$ &  $z^3$ &  $z^3$ &  $z^3$ &  $z^3$ &  $z^3$ &  $z^3$ &  $z^3$ &  $z^3$ &  $z^3$ &  $z^3$ &  $z^3$ &  $z^3$\\
\occupied\occupied\covered  & 0   &  $z^2$ & 0   &  $z^2$ &  $z^2$ &  $z^2$ &  $z^2$ &  $z^2$ & 0   &  $z^2$ &  $z^2$ &  $z^2$ & 0   &  $z^2$ &  $z^2$\\
\occupied\covered \occupied & 0   & 0   &  $z^2$ &  $z^2$ &  $z^2$ &  $z^2$ &  $z^2$ &  $z^2$ & 0   & 0   &  $z^2$ &  $z^2$ &  $z^2$ &  $z^2$ &  $z^2$\\
\occupied\covered \covered  & 0   & 0   & 0   &  $z$  &  $z$  &  $z$  &  $z$  &  $z$  & 0   & 0   &  $z$  &  $z$  & 0   &  $z$  &  $z$ \\
\covered \uncovered \uncovered  & 0   & 0   & 0   & 0   & 0   & 0   & 0   &  1  & 0   & 0   & 0   & 0   & 0   & 0   & 0  \\
\covered \uncovered \covered  & 0   & 0   & 0   & 0   & 0   & 0   &  1  & 0   & 0   & 0   & 0   & 0   & 0   & 0   & 0  \\
\covered \occupied\occupied & 0   & 0   & 0   & 0   &  $z^2$ &  $z^2$ &  $z^2$ &  $z^2$ &  $z^2$ &  $z^2$ &  $z^2$ &  $z^2$ &  $z^2$ &  $z^2$ &  $z^2$\\
\covered \occupied\covered  & 0   & 0   & 0   & 0   &  $z$  &  $z$  &  $z$  &  $z$  & 0   &  $z$  &  $z$  &  $z$  & 0   &  $z$  &  $z$ \\
\covered \covered \uncovered  & 0   & 0   & 0   & 0   & 0   &  1  & 0   & 0   & 0   & 0   & 0   & 0   & 0   & 0   & 0  \\
\covered \covered \occupied & 0   & 0   & 0   & 0   &  $z$  &  $z$  &  $z$  &  $z$  & 0   & 0   &  $z$  &  $z$  &  $z$  &  $z$  &  $z$ \\
\covered \covered \covered  & 0   & 0   & 0   & 0   &  1  & 0   & 0   & 0   & 0   & 0   & 0   & 0   & 0   & 0   & 0  \\
\end{tabular}
}
\caption{The transfer matrix $\overline{A}$ for domination on the cylinder of width $m=3$. Rows and columns are indexed by the new and old signatures $\tau$ and $\sigma$ respectively.}
\label{tab:cyltransfermatrix}
\end{table}

Since the width $m$ of the graph is usually clear from context, we suppress the dependence of $A$ on $m$ in our notation for the most part. Tables \ref{tab:transfermatrix} and \ref{tab:cyltransfermatrix} show the transfer matrices for grids of
width $m=2$ and cylinders of width $m=3$. 

The next two theorems are our key results.

\begin{theorem}
  \label{the:G-transfer}
  Let $\sigma_{\covered}$ be the (cyclic) signature which is covered everywhere,
  $\sigma_{\covered} = (\covered,\covered,\ldots,\covered)$. Then the domination polynomials
  $G_{m\times n}(z)$ and $G_{\overline{m}\times n}(z)$ can be computed
  as
  \begin{subequations}
    \label{eq:domination-transfer}
    \begin{equation}
      \label{eq:G-transfer-grid}
    G_{m\times n}(z) = \sum_{\sigma:\unc(\sigma)=0} (A^n)_{\sigma,\sigma_{\covered}}\,,
  \end{equation}
  \begin{equation}
    \label{eq:G-transfer-cyl}
    G_{\overline{m}\times n}(z) = \sum_{\sigma:\unc(\sigma)=0} (\overline{A}^n)_{\sigma,\sigma_{\covered}}\,.
  \end{equation}
  \end{subequations}
\end{theorem}

\begin{proof}
  Consider $G^\tau_{m\times n}$, the generating function of almost
  dominating sets $S$ in the $m\times n$ grid with signature $\tau$. 
  Now $\tau$ is compatible with multiple signatures $\sigma$
  on row $n-1$. For each such $\sigma$, placing $\tau$ on the $n$th row increases 
  $S$ by $\occupied(\tau)$ and thus multiplies $G^{\sigma}_{m \times (n-1)}$ by a factor $z^{\occupied(\tau)}$. 
  Hence we can write
  \begin{equation}
    \label{eq:transfer-matrix-single-iteration}
    G^\tau_{m\times n} (z) 
    = \sum_{\text{$\sigma: \tau$ compatible}} z^{\occupied(\tau)} G^{\sigma}_{m\times (n-1)}(z) 
    = \sum_{\sigma} A_{\tau,\sigma} G^{\sigma}_{m\times (n-1)}(z)\,,
  \end{equation}
  and applying this reasoning recursively gives
  \begin{equation}
    \label{eq:transfer-matrix-n-iteration}
    G^\tau_{m\times n} (z) = \sum_{\sigma} (A^{n-1})_{\tau,\sigma}
    G^{\sigma}_{m\times 1}(z)\,.
  \end{equation}
  Now, $G^{\sigma}_{m\times 1}$ is a valid generating function for
  an almost dominating set, i.e., for signatures $\sigma$ with weights $z^{\occupied(\sigma)}$, 
  with the additional property that $\sigma$ does not contain
  any $\covered$s. If you look again at Definition \ref{def:compatible}, this is equivalent to saying that $\sigma$ 
  is compatible with $\sigma_{\covered}$. This gives
  \begin{equation}
    \label{eq:transfer-matrix-initial-value}
    G^\sigma_{m\times 1} (z) =  A_{\sigma,\sigma_{\covered}} \, .
  \end{equation}
   Combining this with~\eqref{eq:transfer-matrix-n-iteration} and~\eqref{eq:almost-full}
  completes the proof of \eqref{eq:G-transfer-grid}. The proof of \eqref{eq:G-transfer-cyl} is similar. 
\end{proof}

Domination on the torus is like domination on the cylinder, except
that occupied vertices in the $n$th row can cover vertices in the 1st row
and vice versa. As the following theorem shows, this corresponds to taking 
the trace of the $n$th power of the transfer matrix.

\begin{theorem}
  \label{the:G-transfer-torus}
  Let $G_{\overline{m}\times \overline{n}}(z)$ denote the domination polynomial of the $m\times
  n$ torus. Then
  \begin{equation}
    \label{eq:G-transfer-torus}
    G_{\overline{m}\times \overline{n}} (z) 
    =  \Tr \overline{A}^n \,.
  \end{equation}
\end{theorem}

\begin{proof}
 On the torus, in addition to requiring that the cyclic signature $\sigma$ on the $t$th row is 
 compatible to the signature in the $(t-1)$st row for all $1 < t \le n$, 
 we also need the signature on the 1st row to be compatible with 
 the one on the $n$th row. 
  We can find all such configurations by adding a 0th row to the graph with signature $\sigma$, 
  applying the transfer matrix $n$ times (note that the first application of $\overline{A}$ requires that 
  the signature on the 1st row is compatible with $\sigma$) 
  and picking out the entries of $\overline{A}^n$ where the signature 
  on the $n$th row is also $\sigma$. Thus the generating function for dominating sets on the 
  $m \times n$ torus with signature $\sigma$ on the $n$th row is
  \begin{equation}
    \label{eq:T-transfer-proof}
    G^\sigma_{\overline{m}\times \overline{n}} (z) = (\overline{A}^n)_{\sigma,\sigma} \, ,
  \end{equation}
and summing over all $\sigma$ gives \eqref{eq:G-transfer-torus}.
\end{proof}

Theorems \ref{the:G-transfer} and \ref{the:G-transfer-torus} reduce the problem of computing the domination polynomials  
for the $m\times n$ grid, cylinder, or torus to computing the $n$th power of the transfer matrix
$A$ or $\overline{A}$. There are several ways to do this efficiently. We can compute the $n$th power of a
matrix $A$ by squaring it $\lfloor \log_2 n \rfloor$ times to obtain $A^1, A^2, A^4, \ldots, A^{2^{\lfloor \log_2 n \rfloor}}$, 
and multiplying whichever of these powers correspond to $1$s in the binary expansion of $n$. 
Squaring an $N\times N$ matrix can be done in time $O(N^\omega)$, where $\omega=3$
for the naive schoolbook method or $\omega = 2.371552$ with the
fastest known algorithm \cite{williams:etal:24}.  In our case
$N=a(m)$ or $\overline{a}(m)$, and in both cases $N=O(\lambda^m)$ with $\lambda=1+\sqrt{2}=2.4142\ldots$\ 
This gives a time complexity of essentially $O(\lambda^{\omega m})$. 
In addition, squaring a matrix whose entries are polynomials of degree $nm$ requires us
to multiply such polynomials, which takes $O((nm)^2)$ time 
using the simplest method of adding all the cross-terms. However, rather than 
analyzing the running time of this repeated-squaring approach in detail, we present a faster algorithm 
in the next section.

To conclude this section, we briefly discuss another transfer matrix approach to
domination on the grid. 
Oh \cite{oh:21} proposed what he called the ``state matrix recursion method''
for the domination polynomial of the grid. Although the phrase
``transfer matrix'' does not appear in \cite{oh:21}, it is essentially
a transfer matrix method, and Oh's Theorems 1 and 2 provide 
expressions for $G_{m\times n}(z)$ and $G_{\overline{m}\times n}(z)$ in the 
same spirit as ours. However, Oh focuses on the edges of
the graph rather than the vertices. In a graph with a dominating set, there are four 
types of edges, whose endpoints are labeled 
$(\covered,\covered)$,
$(\covered,\occupied)$, 
$(\occupied,\covered)$ and
$(\occupied,\occupied)$. 
Since each of the $m$ vertical edges connecting one row to the next 
can be in one of these four states, Oh's transfer matrix is $4^m$-dimensional rather than $O(\lambda^m)$-dimensional, 
making it less efficient than our transfer matrix to compute
domination polynomials. On the other hand, Oh's transfer matrix 
can be computed by a surprisingly simple recurrence, giving it an elegant mathematical form.

\section{The algorithm}
\label{sec:algorithm}

\begin{figure}
  \centering
  \includegraphics[width=0.6\linewidth]{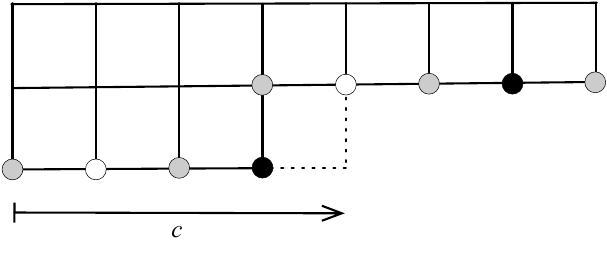}
  \caption{The algorithm adds a new row one vertex at a time from left to right. Here we illustrate a step where we add a vertex in column $c=5$ at the dashed lines. The current signature is $\sigma=\text{\covered \uncovered \covered \occupied'\uncovered \covered \occupied \covered}$ where ' marks the ``kink.'' Whether the new vertex is unoccupied or occupied produces one of two new signatures, 
  $\sigma_0 = \text{\covered \uncovered \covered \occupied \covered'\covered \occupied \covered}$ or
  $\sigma_1 = \text{\covered \uncovered \covered \occupied \occupied'\covered \occupied \covered}$. 
  However, in this example $\sigma[c] = \uncovered$ (the uncovered vertex above the new vertex) so the new vertex 
  must be occupied and $\sigma_0$ is invalid.}
  \label{fig:kink}
\end{figure}

The repeated-squaring approach for computing the $n$th power of the transfer matrix, described in the previous section, 
fails to take advantage of $A$'s (and $\overline{A}$'s) structure. First of all, these matrices are quite sparse, since most pairs of signatures are not compatible. 
Secondly, their nonzero entries are powers of $z$; so, rather than multiplying arbitrary polynomials, 
we can multiply by $A$ or $\overline{A}$ by shifting the coefficients of each polynomial and adding the results. 
Thirdly, and most importantly, in this section we will show how to add a row, and thus multiply by $A$, 
using a series of even simpler operations. 
This will reduce the running time from $O(\lambda^{\omega m})$ 
to essentially $O(\lambda^m)$.

From a bird's eye perspective, the transfer matrix method turns a
two-dimensional problem into a sequence of $n$ one-dimensional problems.
This idea can be applied again. By filling the new row one vertex at a time, from left to right, 
we can subdivide the one-dimensional problem of adding a row 
into a sequence of $m$ zero-dimensional problems. 

At each step, the $n$th row is filled up to column $c-1$. The corresponding signature contains
a ``kink'' at column $c$, where it hops up to the $(n-1)$st row; see Figure~\ref{fig:kink}. 
When we add a vertex in the $c$th column, this signature can be mapped to two possible signatures, 
depending on whether this new vertex is occupied or not.

Across the kink at $c$, the substrings (\uncovered,\occupied) and
(\occupied,\uncovered) are no longer forbidden, so the number of
signatures is a bit larger than $a(m)$ or $\overline{a}(m)$. And in
the case of the non-cyclic signatures, the number
depends on $c$. However, we show in Appendix~\ref{sec:n_signatures} that for each $1 \le c \le m$ 
the number of signatures grows as $O(\lambda^m)$. 

\begin{figure}
  \centering
  \parbox{0.68\linewidth}{
  \begin{algorithmic}[0]
    \State \textbf{subroutine extend} $(\sigma, c)$
    \State $\sigma_1 \defeq \sigma$
    \State $\sigma_1[c] \defeq \occupied$
    \If {$c>1$ \textbf{and}  $\sigma_1[c-1] = \uncovered$}
    \State $\sigma_1[c-1]\defeq\covered$
    \EndIf
    \If {$\sigma[c] = \uncovered$} \Comment{new vertex needs to cover $\sigma[c]$}
    \State $\sigma_0 \defeq \text{invalid}$
    \Else
    \State $\sigma_0 \defeq \sigma$
    \If{$\sigma[c] = \occupied$}
       \State $\sigma_0[c] \defeq \covered$
    \ElsIf{$c > 0$ \textbf{and} $\sigma_0[c-1] = \occupied$}
       \State $\sigma_0[c] \defeq \covered$
    \Else
    \State $\sigma_0[c] \defeq \uncovered$
    \EndIf
    \EndIf
    \State \Return $\sigma_0, \sigma_1$
\end{algorithmic}}
\caption{Adding a new vertex at column $c$ in the current row.}
\label{fig:extend}
\end{figure}

Each step of this new transfer matrix algorithm is the addition of a new
vertex in a partially filled or empty row. This gives a subroutine
$\textbf{extend}(\sigma, c)$, which we show for the grid in Figure~\ref{fig:extend}. 
This subroutine interprets $\sigma$ as a signature where the current row
is filled up to the $(c-1)$st column, and returns up to two signatures $\sigma_0,\sigma_1$ where
$\sigma_0$ (resp., $\sigma_1$) results from $\sigma$ by adding an unoccupied
(resp., occupied) vertex in column $c$. 

The $\textbf{extend}$ subroutine takes care of the
compatibility between $\sigma$, $\sigma_0$, and $\sigma_1$. If the vertex to the left 
of the new vertex is uncovered in $\sigma$, in $\sigma_1$ it becomes covered by the new occupied vertex. 
It also marks the new vertex as covered in $\sigma_0$ if the vertex to its left or above it is occupied in $\sigma$. 
Finally, if $\sigma[c] = \uncovered$, i.e., if the vertex immediately above the new vertex is uncovered 
as in Figure~\ref{fig:kink}, then the new vertex must be occupied. In that case $\sigma_0$ is defined as \emph{invalid}, 
and does not need to be pursued further by the algorithm. 

\begin{figure}
  \centering
  \parbox{0.68\linewidth}{
  \begin{algorithmic}[0]
    \State $L_{\text{old}} \defeq (\sigma_{\covered}, 1)$ \Comment{zeroth row configuration, $\sigma_{\covered}=(\covered,\covered,\ldots,\covered)$}
  \For{$r = 1,\ldots,n$} 
     \State $L_{\text{new}} \defeq \text{empty list}$
  \For{$c = 1,\ldots,m$} 
  \While{$L_{\text{old}}$ not empty} 
  \State take $(\sigma,G^\sigma)$ out of $L_{\text{old}}$
  \State $\sigma_0, \sigma_1 \defeq \textbf{extend}(\sigma, c)$
  \If{$\sigma_0 = \text{invalid}$} ignore $\sigma_0$
  \ElsIf{$(\sigma_0, \cdot) \not\in L_{\text{new}}$} add $(\sigma_0, G^\sigma)$ to $L_{\text{new}}$
  \Else\ replace $(\sigma_0, G)$ in $L_{\text{new}}$ with $(\sigma_0,G+G^\sigma)$
  \EndIf
  \If{$(\sigma_1, \cdot) \not\in L_{\text{new}}$} 
    add $(\sigma_1, z G^\sigma)$ to $L_{\text{new}}$
  \Else\ replace $(\sigma_1, G)$ in $L_{\text{new}}$ with $(\sigma_1,G+z G^\sigma)$
   \EndIf
  \EndWhile
  \State $L_{\text{old}} \defeq L_{\text{new}}$ 
  \EndFor
  \EndFor
  \State $\displaystyle G_{n,m}(z) \defeq \sum_{\mathclap{ \sigma \in
    L_{\text{old}},\, \uncovered(\sigma) = 0}} G^\sigma(z)$
\end{algorithmic}}
\caption{The algorithm to compute the domination polynomial of the grid 
  $G_{m\times n}$. Completing each row, i.e., completing the inner loop over the $m$ columns, has the effect of multiplying by the transfer matrix $A$. As a programming detail, the assignment $L_{\text{old}} \defeq L_{\text{new}}$ is by reference 
  (i.e., by moving a pointer) to avoid copying data from one location in memory to another.}
\label{fig:loop}
\end{figure}

We use the subroutine \textbf{extend} in an algorithm that loops over the 
$n$ rows and $m$ columns of the grid (Figure~\ref{fig:loop}). 
This algorithm builds the rows of the grid one vertex at a time  
while maintaining a list 
of configurations, i.e., pairs $(\sigma, G^\sigma)$ where $\sigma$ is a 
signature and $G^\sigma$ is the corresponding generating function. 
Whenever we add a new vertex we multiply 
by $z$ if that vertex is occupied, adding $G^\sigma$ to $G^{\sigma_0}$ 
and adding $zG^\sigma$ to $G^{\sigma_1}$. 
Note that $\sigma_0$ or $\sigma_1$ might already be in the list 
of new signatures, since adding the new vertex hides the vertex above it. 
That is, $\sigma_0 = \sigma'_0$ if $\sigma$ and $\sigma'$ differ only in 
column $c$, and similarly for $\sigma_1$ and $\sigma'_1$. 
Each loop where $c$ ranges from $1$ to $m$ thus adds a new row and 
effectively applies the transfer matrix. This continues until we complete the $n$th row 
and obtain the dominating polynomial for the entire grid.

In order to carry out these computations for large grids, it turns out that memory, 
not time, is the limiting resource. Thus to reach grids as large as possible, we need to think carefully 
about how to represent and store both signatures $\sigma$ and their polynomials $G^\sigma$ as efficiently as possible. 
To some readers the rest of this section will seem like mere implementation details. But these details play an essential role. 
While both the time and space requirements of our algorithm are exponential, they reduce the exponent, 
and without them we would have no hope of obtaining the results we present in the next section.

First, to represent the signatures $\sigma$, we treat the three symbols \uncovered, \covered, and \occupied\ as ternary digits, 
and interpret each $\sigma$ as an integer between $0$ and $3^m-1$. 
Since $3^{40} \le 2^{64}$, the signatures fit into 64-bit integers as long as $m \le 40$. 

We store the polynomials $G^\sigma$ as vectors of integer coefficients. However, 
since these coefficients grow exponentially in $mn$, they 
quickly get too large to store as fixed-width integers with $32$, $64$, or $128$ bits.
One could use variable-length integers to deal with this problem,
but this would add a factor $nm$ to both the time and the space complexity. 

Instead, we stick with fixed-width integers and
use modular arithmetic. For some integer $b$, there is a set of prime moduli 
$p_i < 2^b$ such that $\prod_i p_i \geq 2^{mn}$. 
We then carry out our calculations mod $p_i$ using $b$-bit integers, 
and use the Chinese Remainder Theorem
\cite{clr:algorithms} to recover the coefficients. Even for our largest computations, integers of length $b = 16$ suffice. 
This approach trades space (the length of the integers) for time 
(one run for each prime modulus). But the runs for different moduli can be done in parallel, 
and we do them on separate processors. 
The final computation using the Chinese Remainder Theorem has to be done with 
variable-length integers to produce the coefficients of $G^\sigma$, but this takes  
time and space which is polynomial in their length $mn$. 

To contain lists of configurations of exponential size, an efficient data structure
is mandatory. There is no point in using tables of size $3^m$ (the number of possible 
ternary sequences) when only $O(\lambda^m)$ signatures actually appear. 
We use an ordered associative container like \texttt{set} or \texttt{map} from
the standard C++ library where signatures are ordered according to their ternary value. 
These data structures guarantee logarithmic complexity for search, insert and delete
operations, so for lists of size exponential in $m$ they work in $O(m)$ time \cite{clr:algorithms}.   

Since the maximum degree of $G^\sigma$ is $mn$, our integers have fixed width $b$, and there are 
$O(\lambda^m)$ different signatures $\sigma$, the total space complexity of the algorithm is $O(mn\lambda^m)$.  
The time complexity of our algorithm is $O(m^3 n^2 \lambda^m)$. The
factor $\lambda^m$ comes from the size of the lists. One factor of $mn$ comes from the loops in Figure~\ref{fig:loop} 
that add the $mn$ vertices one at a time. Another factor of $mn$ comes from copying the polynomials $G^\sigma$, shifting them 
(i.e., multiplying them by $z$) and adding them together.  
The last factor of $m$ comes from the logarithmic complexity of the list operations. Ignoring polynomial factors, then, 
our time and space complexity is $O(\lambda^m)$. 

What we have explained so far is the algorithm for the grid. The cylinder
requires only a small change in the subroutine \textbf{extend}.
When adding the last vertex of a row at $c=m$, the subroutine has to
ensure compatibility with the vertex at $c=1$ to make the full signature
cyclic. Other than that, no changes are required, and the space and
time complexity is the same as for the grid. In particular the algorithm 
shown in Figure~\ref{fig:loop} stays the same. 

The time complexity increases, however, when we adapt our algorithm to the
torus. Here we have to run the algorithm of Figure~\ref{fig:loop} 
for each cyclic signature $\sigma$ in the zeroth row, instead of just 
starting with $\sigma_{\covered}$.  
Hence we need an additional outer loop of length
$\overline{a}(m) = O(\lambda^m)$, resulting in an overall time complexity of
$O(m^3 n^2 \lambda^{2m})$, or $O(\lambda^{2m})$ ignoring polynomial factors. The space complexity remains the same.

\begin{figure}
  \centering
  \includegraphics[width=0.6\linewidth]{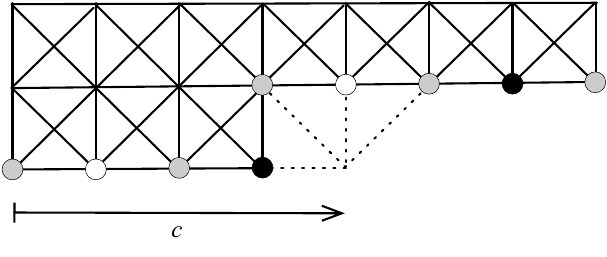}
  \caption{In the king graph, the new vertex at $c$ has to ensure
    compatibility with 4 neighbors.}
  \label{fig:kink-king}
\end{figure}

For the king graph $K_{m \times n}$, we just need to modify the \textbf{extend}
subroutine, since it must consider all four neighbors of the new vertex to ensure
compatibility (Figure~\ref{fig:kink-king}). Because of the
neighbor in the north-west, the number of signatures is now $a(m+1)$. 
The main algorithm in Figure~\ref{fig:loop} stays the same, and the time and space complexity are the same 
as for the grid and the cylinder.

To push our computations further, we take advantage of symmetries.
We can identify each signature of a full row ($c=n$) with its mirror
image, which roughly halves the number of signatures. 
For cyclic signatures we have also translational
symmetry, which reduces the number of cyclic signatures by a factor of
approximately $1/m$. See Appendix~\ref{sec:n_signatures} for the
precise factors.

To give the reader an idea about the actual computational resources needed,
consider the computationally largest task we solved. Using 16-bit
integers for the coefficients, computing
$G_{\overline{24}\times 24}(z) \bmod p_i$ for each prime modulus took 125 hours of 
wall-clock time and required 481 GB of memory. Finally, we needed 36 parallel runs 
for different moduli $p_i$ to recover the coefficients using  
the Chinese Remainder Theorem.

\section{Results}
\label{sec:results}

\begin{table}
  \centering
\begin{tabular}{l|rrc}
  & $m\leq$ & $n\leq$ & $m+n \leq$ \\\hline
  $G_{m\times n}(z)$ & & & $44$ \\
  $G_{\overline{m}\times n}(z)$ & $24$ & $24$ &  \\
  $G_{\overline{m}\times \overline{n}}(z)$ & $17$ & $17$ & \\
  $K_{m\times n}(z)$ &  &  & $44$\\
  $G_{n\times n}(1)$ & & $24$ & \\
  $G_{m\times n}(1)$ & $22$ & $100$ & \\
  $G_{\overline{m}\times n}(1)$ & $22$ & $100$ & \\
  $G_{\overline{n}\times n}(1)$ & & $26$ & \\
  $K_{m\times n}(1)$ & $22$ & $100$ &
\end{tabular}
\caption{Sizes of graphs for which we have computed domination polynomials or
  the total number of dominating sets.}
\label{tab:sizes}
\end{table}

Our algorithm allowed us to compute domination polynomials and the
total number of dominating sets for the graphs listed in
Table~\ref{tab:sizes}. We show the complete domination polynomials for examples of size $m=n \le 8$ in Appendix~\ref{sec:showcase}. The complete data is available from the author's website
\cite{my-domination-site}. 

The varying sizes for which we can carry out these computations 
are due to two facts. First, the number of cyclic signatures $\overline{a}(m)$ is less than the number of signatures $a(m)$, 
making computations for the cylinder somewhat easier than those for the grid. 
Secondly, as discussed above, the computation time for the torus has an extra factor of $\lambda^m$ 
due to the need to sum over all starting signatures. 

To compute the total number of dominating sets, we used a
modified version of our algorithm, in which we do not store the full
domination polynomial with each signature, but only its value at
$z=1$. This saves us a factor of $mn$ in time and space complexity and
allows us to solve larger systems. In particular, we calculated the
number of dominating sets for $m \leq 22$ and $n \leq 100$ for the
grid, the cylinder, and the king graph. This allowed us to compute
precise numerical estimates for the growth rate of these graphs (see Section~\ref{sec:growthrate}).

\begin{table}
  \centering
  \begingroup
  \setlength{\tabcolsep}{3pt}
  \footnotesize
  \begin{tabular}{r|rrrrrrrrrrrrrrrrrrrrrrrr}
 & \multicolumn{24}{c}{$m$}\\
 $n$ & 1 & 2 & 3 & 4 & 5 & 6 & 7 & 8 & 9 & 10 & 11 & 12 & 13 & 14 & 15 & 16 & 17 & 18 & 19 & 20 & 21 & 22 & 23 & 24\\\hline
1 & 1 & 1 & 1 & 2 & 2 & 2 & 3 & 3 & 3 & 4 & 4 & 4 & 5 & 5 & 5 & 6 & 6 & 6 & 7 & 7 & 7 & 8 & 8 & 8 \\
2 & 1 & 2 & 2 & 2 & 3 & 4 & 4 & 4 & 5 & 6 & 6 & 6 & 7 & 8 & 8 & 8 & 9 & 10 & 10 & 10 & 11 & 12 & 12 & 12 \\
3 & 1 & 2 & 3 & 3 & 4 & 5 & 6 & 6 & 7 & 8 & 9 & 9 & 10 & 11 & 12 & 12 & 13 & 14 & 15 & 15 & 16 & 17 & 18 & 18 \\
4 & 2 & 3 & 4 & 4 & 6 & 6 & 7 & 8 & 10 & 10 & 11 & 12 & 13 & 14 & 15 & 16 & 17 & 18 & 19 & 20 & 21 & 22 & 23 & 24 \\
5 & 2 & 3 & 4 & 5 & 7 & 8 & 9 & 10 & 12 & 12 & 14 & 15 & 17 & 17 & 19 & 20 & 21 & 22 & 24 & 24 & 26 & 27 & 29 & 29 \\
6 & 2 & 4 & 5 & 6 & 8 & 9 & 11 & 12 & 14 & 15 & 16 & 18 & 20 & 20 & 22 & 24 & 25 & 26 & 28 & 30 & 31 & 32 & 34 & 35 \\
7 & 3 & 4 & 6 & 7 & 9 & 10 & 12 & 14 & 16 & 17 & 19 & 20 & 22 & 24 & 25 & 27 & 29 & 30 & 32 & 34 & 36 & 37 & 39 & 40 \\
8 & 3 & 5 & 7 & 8 & 10 & 12 & 14 & 16 & 18 & 19 & 21 & 23 & 25 & 27 & 29 & 30 & 32 & 34 & 36 & 38 & 40 & 42 & 44 & 46 \\
9 & 3 & 5 & 7 & 9 & 11 & 13 & 15 & 18 & 20 & 21 & 24 & 26 & 28 & 30 & 32 & 34 & 36 & 38 & 41 & 42 & 44 & 46 & 49 & 51 \\
10 & 4 & 6 & 8 & 10 & 12 & 14 & 17 & 20 & 22 & 24 & 26 & 28 & 31 & 33 & 36 & 38 & 40 & 42 & 45 & 47 & 49 & 51 & 54 & 56 \\
11 & 4 & 6 & 9 & 11 & 13 & 16 & 18 & 21 & 24 & 26 & 28 & 31 & 34 & 36 & 39 & 41 & 44 & 46 & 49 & 52 & 54 & 56 & 59 & 62 \\
12 & 4 & 7 & 10 & 12 & 14 & 17 & 20 & 23 & 26 & 28 & 31 & 34 & 37 & 39 & 42 & 45 & 48 & 50 & 53 & 56 & 59 & 61 & 64 & 67 \\
13 & 5 & 7 & 10 & 13 & 15 & 18 & 21 & 25 & 28 & 30 & 33 & 36 & 40 & 42 & 45 & 48 & 51 & 54 & 57 & 60 & 63 & 66 & 69 & 72 \\
14 & 5 & 8 & 11 & 14 & 16 & 20 & 23 & 27 & 30 & 32 & 36 & 39 & 42 & 45 & 48 & 52 & 55 & 58 & 61 & 64 & 68 & 70 & 74 & 77 \\
15 & 5 & 8 & 12 & 15 & 17 & 21 & 24 & 28 & 32 & 34 & 38 & 41 & 45 & 48 & 51 & 55 & 58 & 62 & 65 & 68 & 72 & 75 & 79 & 82 \\
16 & 6 & 9 & 13 & 16 & 18 & 22 & 26 & 30 & 34 & 36 & 40 & 44 & 48 & 51 & 54 & 58 & 62 & 66 & 69 & 72 & 76 & 80 & 84 & 87 \\
17 & 6 & 9 & 13 & 17 & 19 & 24 & 27 & 32 & 36 & 38 & 43 & 46 & 51 & 54 & 57 & 62 & 65 & 70 & 73 & 76 & 81 & 84 & 89 & 92 \\
18 & 6 & 10 & 14 & 18 & 20 & 25 & 29 & 34 & 38 & 40 & 45 & 49 & 54 & 57 & 60 & 65 & 69 & 73 & 77 & 80 & 85 & 89 & 93 & 97 \\
19 & 7 & 10 & 15 & 19 & 21 & 26 & 30 & 36 & 40 & 42 & 47 & 51 & 57 & 60 & 63 & 68 & 72 & 77 & 81 & 84 & 89 & 93 & 98 & 102 \\
20 & 7 & 11 & 16 & 20 & 22 & 28 & 32 & 38 & 42 & 44 & 50 & 54 & 60 & 63 & 66 & 72 & 76 & 81 & 85 & 88 & 94 & 98 & 103 & 107 \\
21 & 7 & 11 & 16 & 21 & 23 & 29 & 33 & 39 & 44 & 46 & 52 & 56 & 62 & 66 & 69 & 75 & 79 & 85 & 89 & 92 & 98 & 102 & 108 & 112 \\
22 & 8 & 12 & 17 & 22 & 24 & 30 & 35 & 41 & 46 & 48 & 54 & 59 & 65 & 69 & 72 & 78 & 83 & 89 & 93 & 96 & 102 & 107 & 113 & 117 \\
23 & 8 & 12 & 18 & 23 & 25 & 32 & 36 & 43 & 48 & 50 & 57 & 61 & 68 & 72 & 75 & 82 & 86 & 92 & 97 & 100 & 107 & 111 & 117 & 122 \\
24 & 8 & 13 & 19 & 24 & 26 & 33 & 38 & 45 & 50 & 52 & 59 & 64 & 71 & 75 & 78 & 85 & 90 & 96 & 101 & 104 & 111 & 116 & 122 & 127 \\
\end{tabular}
 \endgroup
  \caption{Domination numbers $\gamma(G_{\overline{m}\times n})$ of the
    cylinder graph.}
  \label{tab:domination_numbers_cylinder}
\end{table}

From the domination polynomials we can get other parameters like the
domination number $\gamma$ (the minimum cardinality of a dominating
set) and the number of these minimum dominating sets. There are,
however, more efficient algorithms to compute $\gamma$ without
computing the full domination polynomial. For example, Alanko et al.~\cite{alanko:etal:11}
computed $\gamma(G_{m\times n})$ for $m, n \leq 29$.  And in the same
year, Gon\c{c}alves et al.~\cite{goncalves:etal:11} proved the general formula
\begin{equation}
  \label{eq:gamma_grid_formular}
  \gamma(G_{m\times n}) = \left\lfloor\frac{(m+2)(n+2)}{5}\right\rfloor - 4
  \qquad (n,m \geq 16)\,.
\end{equation}
The sequence $\gamma(G_{n\times n})$ is \seqnum{A104519}.

For the cylinder, we did not find any results for the domination number in the literature. Therefore we present 
Table~\ref{tab:domination_numbers_cylinder} obtained with our algorithm, giving $\gamma(G_{\overline{m} \times n})$ 
for all $m, n \le 24$.

For the torus, Shao et al.~\cite{shao:etal:18} computed
$\gamma(G_{\overline{n}\times\overline{n}})$ for $n\leq 24$.
The corresponding sequence from \seqnum{A094087} in the OEIS lists
values up to $n = 27$. 
Crevals and {\"{O}}sterg{\r{a}}rd \cite{crevals:ostergard:18} found formulae for
$\gamma(G_{\overline{m}\times\overline{n}})$ for $m<20$ and arbitrary $n$. 

Finally, for the king graph, no computation is necessary to find $\gamma$. 
Arshad, Hayat, and Jamil \cite{arshad:hayat:jamil:23} showed 
\begin{equation}
  \label{eq:gamma_kings_formular}
  \gamma(K_{m\times n}) = \left\lceil\frac{m}{3}\right\rceil \left\lceil\frac{n}{3}\right\rceil\,. 
\end{equation}
The sequence $\gamma(K_{m\times n}) $ is \seqnum{A075561}.

\begin{table}
  \centering
  \begin{tabular}{l||cc||c}
    & \multicolumn{2}{c||}{OEIS} & this work \\
    & sequence & \# elements & \# elements \\\hline
    $N_{\gamma}(G_{m\times n})$ & \seqnum{A350820} & 276 & 946 \\
    $N_{\gamma}(G_{n\times n})$ & \seqnum{A347632} & 12 & 22 \\
  $N_{\gamma}(G_{\overline{m}\times n})$ &  &  & 300\\
  $N_{\gamma}(G_{\overline{n}\times \overline{n}})$ & \seqnum{A347557}
                               & 8 & 17 \\
    $N_{\gamma}(K_{m\times n})$ & \seqnum{A350815}  & 276 & 946\\
    $N_{\gamma}(K_{n\times n})$ & \seqnum{A347554} & 12 & 22
  \end{tabular}
\caption{The number of minimum dominating sets, OEIS vs.\ our results.}
\label{tab:OEIS-Ngamma}
\end{table}

\begin{table}
  \centering
  \begin{tabular}{rrrrr}
    $n$ & grid $G_{n\times n}$ & cylinder $G_{\overline{n}\times n}$ &
                                                                       torus
                                                                       $G_{\overline{n}\times\overline{n}}$
    & king $K_{n\times n}$\\\hline
    2 & 6 & 6 & 6 & 4 \\
    3 & 10 & 34 & 48 & 1 \\
    4 & 2 & 16 & 40 & 256 \\
    5 & 22 & 320 & 10 & 79 \\
    6  & 288  & 36  & 18  &  1 \\
    7  & 2  & 56  & 686  &  243856 \\
    8  &  52 & 5565  & 129224  &  3600 \\
    9  &  32 & 20196  &  36 &  1 \\
    10  & 4  & 32210  & 10  & 581571283  \\
    11  & 32  & 88  & 6292  &  281585 \\
    12  & 21600  &  121428 & 162  &  1 \\
    13  &  18 &  388284 &  2704 &  2722291223553 \\
    14  &  540360 & 224  &  56 &  32581328 \\
    15  &  34528 & 1489960  &  10 &  1 \\
    16  &  100406 &  12800 &  916736 &   21706368614058886\\
    17  &  70266144 &  251464 & 29327728  &   5112264019\\
    18  &  1380216154 &  2304 &   &   1 \\
    19  &  1682689266 &  36784 &   &   268740319616196074546\\
    20  &  77900162 &  73062090 &   &   1028516654620\\
    21  &  233645826 &  29787744 &   &  1 \\
    22  &  200997249200 &  738959760 &   &   4839916638142874877046813\\
    23  &   &  73600 &   &   \\
    24  &   &  884736 &   &  
  \end{tabular}
  \caption{The number of minimum dominating sets $N_\gamma$ in various $n\times n$ graphs.}
  \label{tab:mindom}
\end{table}

\begin{table}
  \centering
  \begin{tabular}{l||cc||c}
    & \multicolumn{2}{c||}{OEIS} & this work \\
    & sequence & \# elements & \# elements \\\hline
    $G_{m\times n}(1)$ & \seqnum{A218354} & 198 & 946 \\
    $G_{n\times n}(1)$ & \seqnum{A133515} & 15 & 24 \\
     $G_{\overline{m}\times n}(1)$ & \seqnum{A286514} & 91 & 325\\
  $G_{\overline{n}\times n}(1)$ & \seqnum{A286914}
                                 & 12 & 26 \\
    $G_{\overline{n}\times\overline{n}}(1)$ & \seqnum{A303334} & 8 & 17 \\
    $K_{m\times n}(1)$ & \seqnum{A218663}  & 240 & 946\\
    $K_{n\times n}(1)$ & \seqnum{A133791} & 18 & 22
 \end{tabular}
\caption{The total number of dominating sets, OEIS vs.\ our results.}
\label{tab:OEIS-total}
\end{table}

Much less is known about the number $N_\gamma$ of dominating sets of minimum
size $\gamma$ in these graphs. As often, the OEIS is the only source
of knowledge for these sequences. Table~\ref{tab:OEIS-Ngamma} shows
the OEIS results in comparison to our data. Note that the OEIS stores
2-dimensional sequences in linear order read by antidiagonals. Hence
if one knows a 2-dimensional sequence $A_{m,n}$ for all $m+n \leq k$, the
linear sequence contains $k(k-1)/2$ elements.

Table~\ref{tab:mindom} shows our
results for $N_\gamma$ on $n\times n$ grids, cylinders, tori, and king graphs for various $n$. 
Interestingly, all these sequences are  non monotonic. This is most easily
understood for the king graph: whenever $n$ is divisible by $3$, the board can be tiled
by $(n/3)^2$ king's neighborhoods of size $3 \times 3$, 
and the unique minimum dominating set has one king in the center of each tile. 
For other values of $n$, there are many more arrangements of kings to cope with the interactions between them, 
including ``defects'' where the same vertex is covered by more than one king. 

\begin{table}
  \centering
  \rotatebox{90}{
    \begin{minipage}{\textheight}
    {\tiny
  \begin{tabular}{rl}
1 & {1}\\
2 & {11}\\
3 & {291}\\
4 & {28661}\\
5 & {10982565}\\
6 & {16031828359}\\
7 & {89373230342147}\\
8 & {1904212088591018521}\\
9 & {155026375803222057878889}\\
10 & {48225130114674924906540348115}\\
11 & {57322477811272486520770053115140403}\\
12 & {260351257812272076026660518356378279922077}\\
13 & {4518323367029192938323955373627093441598993023433}\\
14 & {299624906403253780837722041979448648614417149864627538623}\\
15 & {75920925315147351643321026644303797291226237802087526929477529215}\\
16 & {73506889909106192805189860657770727246194962991718803056889525904074284073}\\
17 & {271942808316392849194744097645785662983965539959687059259748204907695205283051026009}\\
18 & {3844236368532391866648400256960909541578872186831658555476061284680772657947421051859726699115}\\
19 & {207646953035194635103908536091585837366786682068864583101871208083573054786197048223152185594685772535619}\\
20 &
     {42857222052146105634373349191667819949558103907317621800201148346434914932291721607912087374140992139321095174522717}
    \\
21 & {33799124158494071677924721438330027552326678220943380010289802318691700882663509162403495792505262990597634684089405594230394397}\\
22 &
     {101852067813875434290677808813207242242027941550711384986722241589709222671728912464757199623466641158572318062654315437747338811787784559599}
    \\
23 &
     {1172781831506875248624491128214157705060276306376593017925817288859622363961093172623234050354080607835750132458893872967965324411437914582111375760429975}
    \\
    24 &
         {51599748273445104357547563282063327408978378463540414307289606061038353421228700908049716554024892118952145604818081188778611944983120752181248857829775808570415670777}
  \end{tabular}}
  \caption{The total number of dominating sets in the $n \times n$ grid, i.e., $G_{n \times n}(1)$.}
\label{tab:totals}
\end{minipage}}
\end{table}

As for $N_\gamma$, the OEIS is the only source of knowledge for the total number of dominating sets. 
Table~\ref{tab:OEIS-total} compares the OEIS entries and 
our results, and Table~\ref{tab:totals} shows the total number on square grids, i.e.,  
$G_{n \times n}(1)$, for all $n \le 24$. Results on cylinders, tori, and king graphs are available 
from the author's website.

\section{Growth rates}
\label{sec:growthrate}

The length of the integers in Table~\ref{tab:totals} demonstrates visually that the
total number of dominating sets $G_{n \times n}(1)$ grows exponentially in the area, 
i.e., as $\mu^{n^2}$ for some $\mu$. In fact, it follows from supermultiplicativity 
and Fekete's Lemma that
\begin{equation}
  \label{eq:fekete}
  \lim_{m,n\to\infty} G_{m\times n}(1)^{\frac{1}{mn}} = \sup G_{m\times n}(1)^{\frac{1}{mn}}\,,
\end{equation}
Since $1 \leq G_{m\times n}(1) \leq 2^{mn}$, the supremum is
finite and the limit
\begin{equation}
  \label{eq:def-growth-constant}
  \mu = \lim_{m,n\to\infty} G_{m\times n}(1)^{\frac{1}{mn}} 
\end{equation}
exists. By the same argument, for any fixed $m$ the limit 
\begin{equation}
  \label{eq:def-growth-rate-m}
  \mu_m = \lim_{n\to\infty} G_{m\times n}(1)^{\frac{1}{mn}}
\end{equation}
exists, and that $\lim_{m\to\infty}\mu_m = \mu$. The same
arguments apply to growth rates on the cylinder, torus, and king graph.

\begin{figure}
  \centering
  \includegraphics[width=0.8\linewidth]{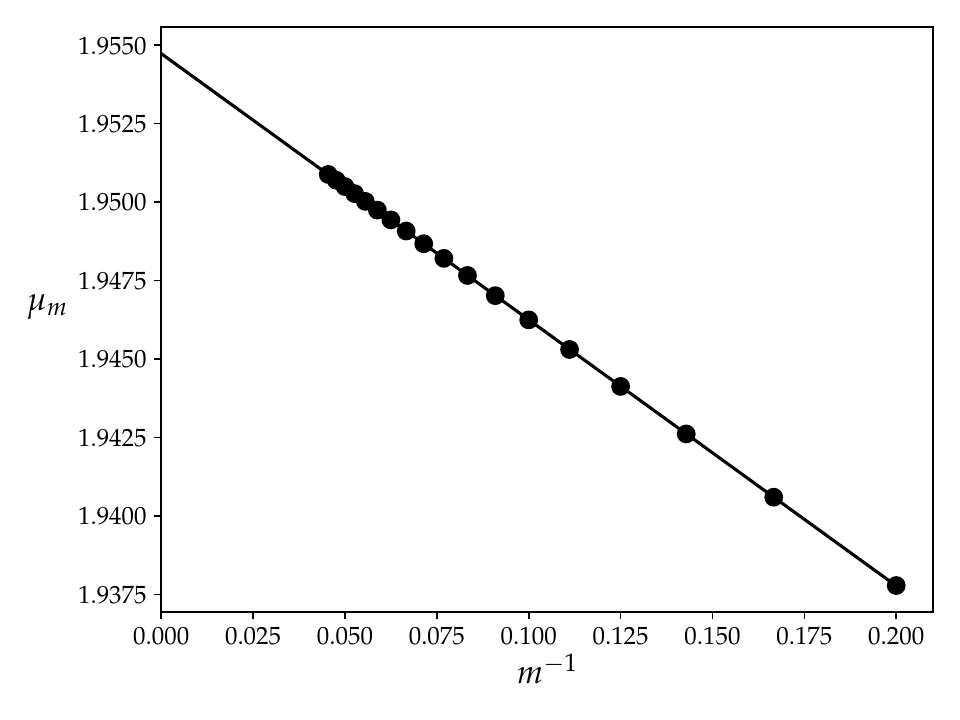}
  \caption{Growth constants $\mu_m$ \eqref{eq:def-growth-rate-m} for the grid versus
    $m^{-1}$.}
  \label{fig:mu_m_grid}
\end{figure}

Thus, in order to estimate $\mu$ numerically, we compute $\mu_m$
for some finite values of $m$ and then extrapolate to $m=\infty$. 
Numerically, we find that the sequence on the right-hand side of
\begin{equation}
  \label{eq:mu_m_from_G}
  \mu_m = \lim_{n\to\infty} \left(\frac{G_{m\times n}(1)}{G_{m\times (n-1)}(1)}\right)^{1/m}
\end{equation}
converges very quickly: the first 50 decimals no longer change for $n > 30$. 
A plot of these estimates of $\mu_m$ as a function of $m^{-1}$
(Figure~\ref{fig:mu_m_grid}) suggests that
\begin{equation}
  \label{eq:mu_m_linear_order}
  \mu_m \simeq \mu + \frac{\mu^{(1)}}{m}
\end{equation}
for some negative constant $\mu^{(1)}$. 

We could use a linear fit in Figure~\ref{fig:mu_m_grid} to estimate
$\mu$. But we proceed more carefully, and take higher order terms
into account.  We assume that
\begin{equation}
  \label{eq:mu_m_all_orders}
  \mu_m = \mu + \sum_{k=1}^\infty \frac{\mu^{(k)}}{m^k}
\end{equation}
and then use Bulirsch-Stoer extrapolation \cite{bulirsch:stoer:64},  a reliable, rapidly converging
method based on rational interpolation. See \cite[Section 4.3]{mertens:22} for a
detailed description of this method applied in a similar situation.
As result we get 
\begin{equation}
  \label{eq:mu_grid_bst}
  \mu = 1.9547511954080(8)\,.
\end{equation}
For the growth rate of the cylinder we get
\begin{equation}
  \label{eq:mu_cyl_bst}
  \overline{\mu} = 1.9547511954085(3)\,,
\end{equation}
which equals the growth rate for the grid within the error bars.  
Based on the assumption that the vertical boundaries of the grid have a decaying effect as 
$m \to \infty$, we conjecture that these two growth rates are in fact equal.

Our data for the torus is not sufficient to compute its growth constant
with the same accuracy, but we do not have to! If you look at
\eqref{eq:G-transfer-cyl} and \eqref{eq:G-transfer-torus}, 
you see that in the limit $n\to\infty$, both right-hand sides are
dominated by the largest eigenvalue of the matrix $\overline{A}$ for $z=1$,
which equals the growth rate $\overline{\mu}_m$. Hence the growth
rate for the torus equals that for the cylinder \eqref{eq:mu_cyl_bst}.

Let $\eta$ denote the growth rate of the king graph. With the same
methods, we estimate
\begin{equation}
  \label{eq:my_kings_bst}
  \eta = 1.997064386596(3)\,.
\end{equation}
This value fits right between the bounds proved by Baumann et al.\ \cite{baumann:calkin:lyle:09},
\begin{equation}
  \label{eq:mu-kings-bounds}
    1.9969 \leq \eta \leq 1.9972 \, ,
\end{equation}
and we conjecture that the first ten decimal digits of \eqref{eq:my_kings_bst} are correct.

\section{Conclusions}
\label{sec:conclusions}

We have presented a transfer matrix algorithm for computing dominating polynomials, and in particular 
counting dominating sets and minimum dominating sets, on the grid, cylinder, and torus graphs, and on 
the king graph. While our algorithm takes exponential time and requires exponential space, we are able 
to significantly reduce the exponent by breaking the induction over rows into an induction over single vertices. 

Along with a careful use of representations and data structures, including representing large integers 
using the Chinese Remainder Theorem, this reduces the running time (ignoring polynomial factors) 
to $O(\lambda^m)$ for the grid, cylinder, and king graph, and $O(\lambda^{2m})$ for the torus, 
where $\lambda = 1+\sqrt{2} = 2.4142...$\ 
We use this algorithm to count dominating sets on these graphs, where the number of rows $n$ 
and columns $m$ range up to 24. This allows us to extend several OEIS sequences considerably, 
and to obtain high-precision estimates of the growth rate $\mu$, where the number of dominating sets 
on $m \times n$ graphs grows asymptotically as $\mu^{mn}$. We believe that similar techniques can 
be applied to many other periodic graphs based on low-dimensional lattices, and to other kinds of sets 
of interest in graph theory.

\section*{Acknowledgments}

I thank J{\"o}rg Schulenburg and the Computing Center of Otto-von-Guericke
University for their support. I am grateful to the medical team of the Clinic for
Hematology and Oncology of Otto-von-Guericke University, where large
parts of this paper have been written. Above all, my gratitude goes to
my friend Cris Moore for his support and inspiration, not only for this project.

\appendix
\section{The number of signatures}
\label{sec:n_signatures}

Let $a_{\uncovered}(m)$, $a_{\covered}(m)$ and $a_{\occupied}(m)$ denote the number of signatures of length $m$ that end with $\uncovered$, $\covered$, and $\occupied$ respectively. Obviously $a(m) = a_{\uncovered}(m) +a_{\covered}(m)+a_{\occupied}(m)$. We also have
\begin{align*}
  a_{\uncovered}(m) &=  a_{\uncovered}(m-1) + a_{\covered}(m-1)\,, \\
  a_{\covered}(m) &=  a_{\uncovered}(m-1) +
                    a_{\covered}(m-1)+a_{\occupied}(m-1) = a(m-1)\,,
  \\
  a_{\occupied}(m) &= a_{\covered}(m-1) + a_{\occupied}(m-1)\,.
\end{align*}
Adding these three equations yields
\begin{displaymath}
  a(m) = 2 a(m-1) + a_{\covered}(m-1)\,,
\end{displaymath}
and inserting the equation for $a_{\covered}(m-1)$ provides us
with the Pell-type recurrence
\begin{equation}
  \label{eq:pell}
  a(m) = 2 a(m-1) + a(m-2)\,.
\end{equation}
The characteristic polynomial of the recurrence is $P(\lambda) =
\lambda^2-2\lambda -1$ with zeroes $1\pm\sqrt{2}$. Hence the
solution of \eqref{eq:pell} is
\begin{subequations}
  \label{eq:pell-solution}
  \begin{equation}
    \label{eq:pell-solution-1}
    a(m) = A_- \left(1-\sqrt{2}\right)^m + A_+\left(1+\sqrt{2}\right)^m\,,
  \end{equation}
  where the coefficients $A_+$ and $A_-$ are fixed by the base cases
  $a(0)$ and $a(1)$,
  \begin{equation}
    \label{eq:pell-solution-2} 
    A_- = \frac{\left(1+\sqrt{2}\right)a(0) -a(1)}{2\sqrt{2}}\qquad
    A_+ = \frac{a(0)-\left(1-\sqrt{2}\right)a(1)}{2\sqrt{2}}
  \end{equation}
\end{subequations}
In our case, $a(1) = 3$ and $a(2)=7$, which implies $a(0) = 1$ and
yields \eqref{eq:a_signatures}.

As we discussed, we can identify a signature with its mirror image. Taking into account this reflection symmetry, 
the resulting number of signatures is
\begin{equation}
      \label{eq:a_row_signatures_reflection}
      \hat{a}(m) = \frac{1}{2}\Big(a(m) + a(\lfloor\textstyle{\frac{m+1}{2}}\rfloor)\Big)\,,
\end{equation}
which is \seqnum{A030270}\,. This formula is easily
understood. Reflection symmetry gives us a factor of $1/2$ for all
non-symmetric signatures. If we apply the factor $1/2$ to all
signatures, we need to add back the number of symmetric signatures,
which are completely specified by their first half.

For signatures with a kink between $c-1$ and $c$, the recurrence reads
\begin{equation}
  \label{eq:pell-c}
  a_c(m) = \begin{cases}
    3 a_c(m-1) & \text{if $m=c$,}\\
    2 a_c(m-1) + a_c(m-2) & \text{otherwise.}
  \end{cases}
\end{equation}
Obviously, $a_c(m)$ follows \eqref{eq:a_signatures} for $m < c$ and
for $m > c$. Hence, $a_c(m)$ is also solved by
\eqref{eq:pell-solution-1}, but with change of $A_-$ and $A_+$ as $m$
passes $c$. The asymptotic scaling $O(\lambda^m)$ persists.

For cyclic signatures, the derivation of \eqref{eq:c_signatures} is a
bit more involved. Knopfmacher et al.\ \cite{knopfmacher:etal:10} used
Chebyshev polynomials to derive the generating function for
$\overline{a}(m)$. Here we give a more elementary derivation. 

Let $\overline{a}_{\sigma_m,\sigma_1}(m)$ denote the number of cyclic signatures
of length $m$ with values $\sigma_1$ and $\sigma_m$ at their $1$st and $m$th position. Then
\begin{equation}
  \label{eq:cyclic-proof-1}
  \begin{aligned}
    \overline{a}_{\uncovered,\uncovered}(m) & = \overline{a}_{\uncovered,\uncovered}(m-1) +
                                   \overline{a}_{\covered,\uncovered}(m-1)\,, \\
    \overline{a}_{\uncovered,\covered}(m) &=
    \overline{a}_{\uncovered,\covered}(m-1)+\overline{a}_{\covered,\covered}(m-1) \,, \\
    \overline{a}_{\covered,\uncovered}(m)  &= \overline{a}_{\uncovered,\uncovered}(m-1) +
    \overline{a}_{\covered,\uncovered}(m-1) \,, \\
    \overline{a}_{\covered,\covered}(m) &= \overline{a}_{\uncovered,\covered}(m-1)+\overline{a}_{\covered,\covered}(m-1)+\overline{a}_{\occupied,\covered}(m-1) \,, \\
    \overline{a}_{\covered,\occupied}(m) &=  \overline{a}_{\covered,\occupied}(m-1) +  \overline{a}_{\occupied,\occupied}(m-1) \,, \\
    \overline{a}_{\occupied,\covered}(m) &=  \overline{a}_{\uncovered,\uncovered}(m-1) + \overline{a}_{\covered,\uncovered}(m-1)+\overline{a}_{\uncovered,\covered}(m-1) + \overline{a}_{\covered,\covered}(m-1) + \overline{a}_{\occupied,\covered}(m-1) \,, \\
    \overline{a}_{\occupied,\occupied}(m) &= \overline{a}_{\covered,\occupied}(m-1) + \overline{a}_{\occupied,\occupied}(m-1) \,.
  \end{aligned}
\end{equation}
On the right-hand sides, $\overline{a}_{\covered,\occupied}$,
$\overline{a}_{\occupied,\covered}$ and $\overline{a}_{\occupied,\occupied}$ appear twice,
and all other $\sigma$'s appear three times. Hence, adding all these equations yields
\begin{equation}
  \label{eq:cyclic-proof-2}
  \overline{a}(m) = 3 \overline{a}(m-1) - \big[\overline{a}_{\covered,\occupied}(m-1)+\overline{a}_{\occupied,\covered}(m-1)+\overline{a}_{\occupied,\occupied}(m-1)\big]\,.
\end{equation}
When we apply the recurrence \eqref{eq:cyclic-proof-1} to the terms in brackets, we
notice that $\overline{a}_{\covered,\occupied}$ and $\overline{a}_{\occupied\occupied}$
appear twice, and all other $\sigma$'s appear exactly once. This gives 
\begin{equation}
  \label{eq:cyclic-proof-3}
  [\cdots] = \overline{a}(m-2) +\big\{ \overline{a}_{\covered,\occupied}(m-2)+\overline{a}_{\occupied\occupied}(m-2)\big\}\,.
\end{equation}
If we apply \eqref{eq:cyclic-proof-1} to the terms in the
curly brackets, we get $\{\cdots\} = \overline{a}(m-3)$, and finally 
\begin{equation}
  \label{eq:cyclic-recurrence}
  \overline{a}(m) = 3 \overline{a}(m-1) - \overline{a}(m-2) -\overline{a}(m-3)\,.
\end{equation}
The characteristic polynomial of this recurrence is
\begin{equation}
  \label{eq:cyclic-characteristic}
  P(\lambda) = \lambda^3 - 3\lambda^2+\lambda + 1 = (\lambda-1)(\lambda^2-2\lambda-1)\,,
\end{equation}
with zeroes $1$, $1-\sqrt{2}$, and $1+\sqrt{2}$. Hence the solution of
\eqref{eq:cyclic-recurrence} is
\begin{equation}
  \label{eq:1}
  \overline{a}(m) = C_1 + C_- \left(1-\sqrt{2}\right)^m + C_+ \left(1+\sqrt{2}\right)^m\,,
\end{equation}
where $C_1$, $C_-$, and $C_+$ depend on the base case $\overline{a}(0)$, $\overline{a}(1)$, 
and $\overline{a}(2)$:
\begin{equation}
  \label{eq:cyclic-base-cases}
  \begin{aligned}
    C_1 &= \frac{1}{2} \overline{a}(0) + \overline{a}(1) - \frac{1}{2} \overline{a}(2)\,, \\
    C_- &= \frac{2+\sqrt{2}}{4\sqrt{2}} \overline{a}(0) -
    \frac{2+2\sqrt{2}}{4\sqrt{2}} \overline{a}(1) + \frac{1}{4} \overline{a}(2)\,, \\
    C_+ &= -\frac{2-\sqrt{2}}{4\sqrt{2}} \overline{a}(0) + \frac{2-2\sqrt{2}}{4\sqrt{2}} \overline{a}(1) +
    \frac{1}{4} \overline{a}(2)\,.
  \end{aligned}
\end{equation}
In our case we have $\overline{a}(1)=3$, $\overline{a}(2)=7$ and $\overline{a}(3)=15$ which implies
$\overline{a}(0)=3$ and therefore $C_1=C_-=C_+=1$, which gives
\eqref{eq:c_signatures}.

If one takes into account circular and reflection symmetry, the number
of signatures is approximately $\overline{a}(m)/2m$, as can be checked
by dividing \seqnum{A208716} by \seqnum{A124696}.

\section{Domination polynomials}
\label{sec:showcase}

Tables \ref{tab:nxn-grid}, \ref{tab:nxn-cylinder}, \ref{tab:nxn-torus}
and \ref{tab:nxn-king} show the domination polynomials of the $n\times
n$ grid, cylinder, torus and king graph for $n\leq 8$. The domination polynomials
for larger and rectangular graphs can be downloaded from the author's website.

\begin{table}
  \centering
  {\small
  \setlength{\tabcolsep}{1pt}
  \begin{tabular}{rcl}
    $G_{1\times 1}(z)$ & $=$ & $z$ \\[0.8ex]
    $G_{2\times 2}(z)$ & $=$ & $\numprint{6}\,z^{2} +\numprint{4}\,z^{3} + z^4$\\[0.8ex]
    $G_{3\times 3}(z)$ & $=$ & $\numprint{10}\,z^3 + \numprint{57}\,z^4+\numprint{98}\,z^5+ \numprint{80}\,z^6+ \numprint{36}\,z^7+\numprint{9}\,z^8+z^9$\\[0.8ex]
    $G_{4\times 4}(z)$ & $=$ & $
                             \numprint{20}\,z^{4} +
                             \numprint{40}\,z^{5} +
                             \numprint{554}\,z^{6} +
                             \numprint{2484}\,z^{7} +
                             \numprint{5494}\,z^{8} +
                             \numprint{7268}\,z^{9} +
                             \numprint{6402}\,z^{10} + $ \\ & & $
                             \numprint{3964}\,z^{11} +
                             \numprint{1760}\,z^{12} +
                             \numprint{556}\,z^{13} +
                             \numprint{120}\,z^{14} +
                             \numprint{16}\,z^{15} +
                             z^{16}$ \\[0.8ex]
    $G_{5\times 5}(z)$ & $=$ & $
                             \numprint{22}\,z^{7} +
                             \numprint{1545}\,z^{8} +
                             \numprint{22594}\,z^{9} +
                             \numprint{140304}\,z^{10} +
                             \numprint{492506}\,z^{11} +
                             \numprint{1126091}\,z^{12} + $ \\ & & $
                             \numprint{1823057}\,z^{13} +
                             \numprint{2204694}\,z^{14} +
                             \numprint{2063202}\,z^{15} +
                             \numprint{1528544}\,z^{16} +
                             \numprint{908623}\,z^{17} + $ \\ & & $
                             \numprint{435832}\,z^{18} +
                             \numprint{168426}\,z^{19} +
                             \numprint{51953}\,z^{20} +
                             \numprint{12550}\,z^{21} +
                             \numprint{2296}\,z^{22} +
                             \numprint{300}\,z^{23} + $ \\ & & $
                             \numprint{25}\,z^{24} +
                             z^{25}$ \\[0.8ex]
    $G_{6\times 6}(z)$ & $=$ & $
                             \numprint{288}\, z^{10} + 
                             \numprint{20896}\, z^{11} +
                             \numprint{478624}\, z^{12} +
                             \numprint{5119512}\, z^{13} +
                             \numprint{32070018}\, z^{14} + $\\ & & $
                             \numprint{133299396}\, z^{15} + 
                             \numprint{397278079}\, z^{16} +
                             \numprint{894777804}\, z^{17} +
                             \numprint{1581325412}\, z^{18} + $\\ & & $
                             \numprint{2254665800}\, z^{19} +
                             \numprint{2648227540}\, z^{20} +
                             \numprint{2602834832}\, z^{21} +
                             \numprint{2165708332}\, z^{22} + $\\ & & $
                             \numprint{1538223528}\, z^{23} +
                             \numprint{937732160}\, z^{24} +
                             \numprint{492091912}\, z^{25} +
                             \numprint{222401360}\, z^{26} + $\\ & & $
                             \numprint{86397060}\, z^{27} +
                             \numprint{28715172}\, z^{28} +
                             \numprint{8101900}\, z^{29} +
                             \numprint{1917814}\, z^{30} +
                             \numprint{374360}\, z^{31} + $\\ & & $
                             \numprint{58757}\, z^{32} +
                             \numprint{7136}\, z^{33} +
                             \numprint{630}\, z^{34} +
                             \numprint{36}\, z^{35} +
                                                                  z^{36} $ \\[0.8ex]
    $G_{7\times 7}(z)$ & $=$ & $
                         \numprint{2}\,z^{12} +
                         \numprint{682}\,z^{13} +
                         \numprint{69818}\,z^{14} +
                         \numprint{2809634}\,z^{15} +
                         \numprint{58346490}\,z^{16} +
                         \numprint{722332499}\,z^{17} + $\\ & & $
                         \numprint{5873091754}\,z^{18} +
                         \numprint{33720209068}\,z^{19} +
                         \numprint{144326231696}\,z^{20} +
                         \numprint{479699210510}\,z^{21} + $\\ & & $
                         \numprint{1277484819726}\,z^{22} +
                         \numprint{2793279785490}\,z^{23} +
                         \numprint{5112738876944}\,z^{24} + $\\ & & $
                         \numprint{7956389260884}\,z^{25} +
                         \numprint{10659803571300}\,z^{26} +
                         \numprint{12421321161300}\,z^{27} + $\\ & & $
                         \numprint{12692372752380}\,z^{28} +
                         \numprint{11448278299084}\,z^{29} +
                         \numprint{9162679913216}\,z^{30} + $\\ & & $
                         \numprint{6533166152352}\,z^{31} +
                         \numprint{4161998104421}\,z^{32} +
                         \numprint{2373420930490}\,z^{33} + $\\ & & $
                         \numprint{1212661131156}\,z^{34} +
                         \numprint{555107862078}\,z^{35} +
                         \numprint{227428059844}\,z^{36} + $\\ & & $
                         \numprint{83222666789}\,z^{37} +
                         \numprint{27112560820}\,z^{38} +
                         \numprint{7828049130}\,z^{39} +
                         \numprint{1990771673}\,z^{40} + $\\ & & $
                         \numprint{442325654}\,z^{41} +
                         \numprint{84949536}\,z^{42} +
                         \numprint{13902582}\,z^{43} +
                         \numprint{1901827}\,z^{44} +
                         \numprint{211672}\,z^{45} + $\\ & & $
                         \numprint{18420}\,z^{46} +
                         \numprint{1176}\,z^{47} +
                         \numprint{49}\,z^{48} +
                         z^{49} $\\[0.8ex]
    $G_{8\times 8}(z)$ & $=$ & $
                         \numprint{52}\,z^{16} +
                         \numprint{15864}\,z^{17} +
                         \numprint{1722568}\,z^{18} +
                         \numprint{88226896}\,z^{19} +
                         \numprint{2530732136}\,z^{20} + $\\ & & $
                         \numprint{45375987524}\,z^{21} +
                         \numprint{550599054884}\,z^{22} +
                         \numprint{4804379992724}\,z^{23} + $\\ & & $
                         \numprint{31600623255338}\,z^{24} +
                         \numprint{162562260288736}\,z^{25} +
                         \numprint{673394654370166}\,z^{26} + $\\ & & $
                         \numprint{2299264864482900}\,z^{27} +
                         \numprint{6594998844457680}\,z^{28} +
                         \numprint{16140569091024412}\,z^{29} + $\\ & & $
                         \numprint{34145122808773410}\,z^{30} +
                         \numprint{63119173723897716}\,z^{31} +
                         \numprint{102895753969864066}\,z^{32} + $\\ & & $
                         \numprint{149077597217535156}\,z^{33} +
                         \numprint{193230536934785376}\,z^{34} +
                         \numprint{225335102676614928}\,z^{35} + $\\ & & $
                         \numprint{237544411406921016}\,z^{36} +
                         \numprint{227287805873540304}\,z^{37} +
                         \numprint{198057834976389932}\,z^{38} + $\\ & & $
                         \numprint{157618769172704668}\,z^{39} +
                         \numprint{114817612849042346}\,z^{40} +
                         \numprint{76694678728213904}\,z^{41} + $\\ & & $
                         \numprint{47038041070108638}\,z^{42} +
                         \numprint{26511846459068480}\,z^{43} +
                         \numprint{13738205846668894}\,z^{44} + $\\ & & $
                         \numprint{6545243405852040}\,z^{45} +
                         \numprint{2865791004809792}\,z^{46} +
                         \numprint{1152143554074948}\,z^{47} + $\\ & & $
                         \numprint{424740089888210}\,z^{48} +
                         \numprint{143310533096044}\,z^{49} +
                         \numprint{44147026143576}\,z^{50} + $\\ & & $
                         \numprint{12377560349296}\,z^{51} +
                         \numprint{3146185878694}\,z^{52} +
                         \numprint{721528535044}\,z^{53} + $\\ & & $
                         \numprint{148407392344}\,z^{54} +
                         \numprint{27176088292}\,z^{55} +
                         \numprint{4389826708}\,z^{56} +
                         \numprint{618261932}\,z^{57} + $\\ & & $
                         \numprint{74786314}\,z^{58} +
                         \numprint{7615724}\,z^{59} +
                         \numprint{635108}\,z^{60} +
                         \numprint{41660}\,z^{61} +
                         \numprint{2016}\,z^{62} +
                         \numprint{64}\,z^{63} +
                         z^{64}$
  \end{tabular}}
  \caption{Domination polynomials of the grid graph $G_{n\times n}$.}
  \label{tab:nxn-grid}
\end{table}

\begin{table}
  \centering
  {\small
  \setlength{\tabcolsep}{1pt}
  \begin{tabular}{rcl}
      $G_{\overline{1}\times 1}(z)$ & $=$ & $z$ \\[0.8ex]
    $G_{\overline{2}\times 2}(z)$ & $=$ & $\numprint{6}\,z^{2} +\numprint{4}\,z^{3} +
                         z^{4}$\\[0.8ex]
    $G_{\overline{3} \times 3}(z)$ & $=$ & $\numprint{34}\,z^{3} +
\numprint{99}\,z^{4} +
\numprint{120}\,z^{5} +
\numprint{84}\,z^{6} +
\numprint{36}\,z^{7} +
\numprint{9}\,z^{8} +
                                z^{9} $\\[0.8ex]
    $G_{\overline{4}\times 4}(z)$ & $=$ & $
                               \numprint{16}\,z^{4} +
\numprint{248}\,z^{5} +
\numprint{1560}\,z^{6} +
\numprint{4752}\,z^{7} +
\numprint{8308}\,z^{8} +
\numprint{9376}\,z^{9} +
\numprint{7404}\,z^{10} +
\numprint{4264}\,z^{11} +
\numprint{1812}\,z^{12} + $\\ & & $
\numprint{560}\,z^{13} +
\numprint{120}\,z^{14} +
\numprint{16}\,z^{15} +
z^{16}$
    \\[0.8ex]
    $G_{\overline{5}\times 5}(z)$ & $=$ & $
                               \numprint{320}\,z^{7} +
\numprint{8525}\,z^{8} +
\numprint{77240}\,z^{9} +
\numprint{354768}\,z^{10} +
\numprint{1000860}\,z^{11} +
\numprint{1934895}\,z^{12} +
\numprint{2744825}\,z^{13} + $\\ & & $
\numprint{2988230}\,z^{14} +
\numprint{2571838}\,z^{15} +
\numprint{1783400}\,z^{16} +
\numprint{1007095}\,z^{17} +
\numprint{464780}\,z^{18} +
\numprint{174710}\,z^{19} +$\\ & & $
\numprint{52905}\,z^{20} +
\numprint{12640}\,z^{21} +
\numprint{2300}\,z^{22} +
\numprint{300}\,z^{23} +
\numprint{25}\,z^{24} +
z^{25}$
    \\[0.8ex]
    $G_{\overline{6}\times 6}(z)$ & $=$ & $
                               \numprint{36}\,z^{9} +
\numprint{5304}\,z^{10} +
\numprint{182640}\,z^{11} +
\numprint{2674472}\,z^{12} +
\numprint{20888976}\,z^{13} +
\numprint{102474888}\,z^{14} +
\numprint{349290996}\,z^{15} +$\\ & & $
\numprint{883272549}\,z^{16} +
\numprint{1733585388}\,z^{17} +
\numprint{2727960890}\,z^{18} +
\numprint{3525246624}\,z^{19} +
\numprint{3808843866}\,z^{20} +$\\ & & $
\numprint{3487178896}\,z^{21} +
\numprint{2732164086}\,z^{22} +
\numprint{1844521704}\,z^{23} +
\numprint{1077669852}\,z^{24} +
\numprint{545975556}\,z^{25} +$\\ & & $
\numprint{239780520}\,z^{26} +
\numprint{91042704}\,z^{27} +
\numprint{29727648}\,z^{28} +
\numprint{8277408}\,z^{29} +
\numprint{1941108}\,z^{30} +
\numprint{376584}\,z^{31} +$\\ & & $
\numprint{58893}\,z^{32} +
\numprint{7140}\,z^{33} +
\numprint{630}\,z^{34} +
\numprint{36}\,z^{35} +
z^{36}
    $\\[0.8ex]
    $G_{\overline{7}\times 7}(z)$ & $=$ & $
                               \numprint{56}\,z^{12} +
\numprint{17878}\,z^{13} +
\numprint{1155252}\,z^{14} +
\numprint{31054898}\,z^{15} +
\numprint{456455958}\,z^{16} +
\numprint{4228396193}\,z^{17} +$\\ & & $
\numprint{27003670764}\,z^{18} +
\numprint{126567019852}\,z^{19} +
\numprint{455787743684}\,z^{20} +
\numprint{1305495024212}\,z^{21} +$\\ & & $
\numprint{3054799279140}\,z^{22} +
\numprint{5964099864170}\,z^{23} +
\numprint{9880494881782}\,z^{24} +
\numprint{14079356852554}\,z^{25} +$\\ & & $
\numprint{17447648954876}\,z^{26} +
\numprint{18972152485706}\,z^{27} +
\numprint{18232693610636}\,z^{28} +
\numprint{15575358475348}\,z^{29} +$\\ & & $
\numprint{11880424274852}\,z^{30} +
\numprint{8119023303202}\,z^{31} +
\numprint{4982943200557}\,z^{32} +
\numprint{2750423714766}\,z^{33} +$\\ & & $
\numprint{1366055406058}\,z^{34} +
\numprint{610263826646}\,z^{35} +
\numprint{244883991996}\,z^{36} +
\numprint{88057328933}\,z^{37} +$\\ & & $
\numprint{28275236934}\,z^{38} +
\numprint{8068294570}\,z^{39} +
\numprint{2032827433}\,z^{40} +
\numprint{448443744}\,z^{41} +
\numprint{85669472}\,z^{42} +$\\ & & $
\numprint{13968430}\,z^{43} +
\numprint{1906219}\,z^{44} +
\numprint{211862}\,z^{45} +
\numprint{18424}\,z^{46} +
\numprint{1176}\,z^{47} +
\numprint{49}\,z^{48} +
z^{49}
    $\\[0.8ex]
    $G_{\overline{8}\times 8}(z)$ & $=$ & $
                               \numprint{5556}\,z^{16} +
\numprint{877312}\,z^{17} +
\numprint{53209280}\,z^{18} +
\numprint{1705112768}\,z^{19} +
\numprint{33445432384}\,z^{20} +
\numprint{439072279040}\,z^{21} +$\\ & & $
\numprint{4109617399080}\,z^{22} +
\numprint{28780589281584}\,z^{23} +
\numprint{156652617731416}\,z^{24} +
\numprint{683114966762944}\,z^{25} +$\\ & & $
\numprint{2445690796232104}\,z^{26} +
\numprint{7333807159180640}\,z^{27} +
\numprint{18724721152985788}\,z^{28} +$\\ & & $
\numprint{41265837337782160}\,z^{29} +
\numprint{79400630946848664}\,z^{30} +
\numprint{134680399945312528}\,z^{31} +$\\ & & $
\numprint{203039926797499914}\,z^{32} +
\numprint{273950585370935584}\,z^{33} +
\numprint{332770579433142856}\,z^{34} +$\\ & & $
\numprint{365749751152851088}\,z^{35} +
\numprint{365293505626221476}\,z^{36} +
\numprint{332720567077905776}\,z^{37} +$\\ & & $
\numprint{277203692560942216}\,z^{38} +
\numprint{211771844116575568}\,z^{39} +
\numprint{148641968502148908}\,z^{40} +$\\ & & $
\numprint{96000555048304144}\,z^{41} +
\numprint{57112559682929880}\,z^{42} +
\numprint{31318418200248960}\,z^{43} +$\\ & & $
\numprint{15833769466628176}\,z^{44} +
\numprint{7379242217245312}\,z^{45} +
\numprint{3168290754707192}\,z^{46} +$\\ & & $
\numprint{1251914193916144}\,z^{47} +
\numprint{454574372292346}\,z^{48} +
\numprint{151368545763424}\,z^{49} +$\\ & & $
\numprint{46103561935240}\,z^{50} +
\numprint{12802119434064}\,z^{51} +
\numprint{3227917903348}\,z^{52} +
\numprint{735359555024}\,z^{53} +$\\ & & $
\numprint{150440930640}\,z^{54} +
\numprint{27431963344}\,z^{55} +
\numprint{4416833096}\,z^{56} +
\numprint{620587536}\,z^{57} +
\numprint{74943232}\,z^{58} +$\\ & & $
\numprint{7623504}\,z^{59} +
\numprint{635360}\,z^{60} +
\numprint{41664}\,z^{61} +
\numprint{2016}\,z^{62} +
\numprint{64}\,z^{63} +
z^{64}
    $\\[0.8ex]
  \end{tabular}}
  \caption{Domination polynomials of the cylinder graph
    $G_{\overline{n}\times n}$.}
  \label{tab:nxn-cylinder}
\end{table}

\begin{table}
  \centering
  {\small
  \setlength{\tabcolsep}{1pt}
  \begin{tabular}{rcl}
    $G_{\overline{1}\times\overline{1}}(z)$ & $=$ & $z$ \\[0.8ex]
    $G_{\overline{2}\times\overline{2}}(z)$ & $=$ & $
                               \numprint{6}\,z^{2} +
                               \numprint{4}\,z^{3} +
                               z^{4}
                               $\\[0.8ex]
    $G_{\overline{3}\times\overline{3}}(z)$ & $=$  & $
                                \numprint{48}\,z^{3} +
                                \numprint{117}\,z^{4} +
                                \numprint{126}\,z^{5} +
                                \numprint{84}\,z^{6} +
                                \numprint{36}\,z^{7} +
                                \numprint{9}\,z^{8} +
                                z^{9}
                                $\\[0.8ex]
    $G_{\overline{4}\times\overline{4}}(z)$ & $=$  & $
                                \numprint{40}\,z^{4} +
                                \numprint{560}\,z^{5} +
                                \numprint{2736}\,z^{6} +
                                \numprint{6800}\,z^{7} +
                                \numprint{10310}\,z^{8} +
                                \numprint{10560}\,z^{9} +
                                \numprint{7832}\,z^{10} +
                                \numprint{4352}\,z^{11} +
                                \numprint{1820}\,z^{12} +$\\ & & $
                                \numprint{560}\,z^{13} +
                                \numprint{120}\,z^{14} +
                                \numprint{16}\,z^{15} +
                                z^{16}
    $\\[0.8ex]
    $G_{\overline{5}\times\overline{5}}(z)$ & $=$  & $
                                \numprint{10}\,z^{5} +
                                \numprint{200}\,z^{6} +
                                \numprint{3050}\,z^{7} +
                                \numprint{31525}\,z^{8} +
                                \numprint{188700}\,z^{9} +
                                \numprint{677690}\,z^{10} +
                                \numprint{1610700}\,z^{11} +
                                \numprint{2740775}\,z^{12} +$\\ & & $
                                \numprint{3527075}\,z^{13} +
                                \numprint{3562700}\,z^{14} +
                                \numprint{2895610}\,z^{15} +
                                \numprint{1923600}\,z^{16} +
                                \numprint{1053175}\,z^{17} +
                                \numprint{475950}\,z^{18} +$\\ & & $
                                \numprint{176600}\,z^{19} +
                                \numprint{53105}\,z^{20} +
                                \numprint{12650}\,z^{21} +
                                \numprint{2300}\,z^{22} +
                                \numprint{300}\,z^{23} +
                                \numprint{25}\,z^{24} +
                                z^{25}
    $\\[0.8ex]
    $G_{\overline{6}\times\overline{6}}(z)$ & $=$  & $
                                \numprint{18}\,z^{8} +
                                \numprint{792}\,z^{9} +
                                \numprint{42480}\,z^{10} +
                                \numprint{901692}\,z^{11} +
                                \numprint{9417660}\,z^{12} +
                                \numprint{57622212}\,z^{13} +
                                \numprint{234273096}\,z^{14} +$\\ & & $
                                \numprint{686972304}\,z^{15} +
                                \numprint{1535339241}\,z^{16} +
                                \numprint{2718976500}\,z^{17} +
                                \numprint{3925148718}\,z^{18} +
                                \numprint{4717557288}\,z^{19} +$\\ & & $
                                \numprint{4795710066}\,z^{20} +
                                \numprint{4172271408}\,z^{21} +
                                \numprint{3133155636}\,z^{22} +
                                \numprint{2042728812}\,z^{23} +
                                \numprint{1160244930}\,z^{24} +$\\ & & $
                                \numprint{574802640}\,z^{25} +
                                \numprint{248126706}\,z^{26} +
                                \numprint{93014644}\,z^{27} +
                                \numprint{30098664}\,z^{28} +
                                \numprint{8330940}\,z^{29} +
                                \numprint{1946676}\,z^{30} +$\\ & & $
                                \numprint{376956}\,z^{31} +
                                \numprint{58905}\,z^{32} +
                                \numprint{7140}\,z^{33} +
                                \numprint{630}\,z^{34} +
                                \numprint{36}\,z^{35} +
                                z^{36}
    $\\[0.8ex]
    $G_{\overline{7}\times\overline{7}}(z)$ & $=$  & $
                                \numprint{686}\,z^{12} +
                                \numprint{205996}\,z^{13} +
                                \numprint{9203432}\,z^{14} +
                                \numprint{182205912}\,z^{15} +
                                \numprint{2082222660}\,z^{16} +
                                \numprint{15633666139}\,z^{17} +$\\ & & $
                                \numprint{83589101666}\,z^{18} +
                                \numprint{336543504122}\,z^{19} +
                                \numprint{1062883834964}\,z^{20} +
                                \numprint{2715977010936}\,z^{21} +$\\ & & $
                                \numprint{5751616552262}\,z^{22} +
                                \numprint{10287521966512}\,z^{23} +
                                \numprint{15778748654928}\,z^{24} +
                                \numprint{21007961215738}\,z^{25} +$\\ & & $
                                \numprint{24521234114524}\,z^{26} +
                                \numprint{25294410442980}\,z^{27} +
                                \numprint{23207364109062}\,z^{28} +
                                \numprint{19035405413402}\,z^{29} +$\\ & & $
                                \numprint{14013460448554}\,z^{30} +
                                \numprint{9286179999558}\,z^{31} +
                                \numprint{5549897026821}\,z^{32} +
                                \numprint{2994639956448}\,z^{33} +$\\ & & $
                                \numprint{1459111542322}\,z^{34} +
                                \numprint{641506327014}\,z^{35} +
                                \numprint{254073916530}\,z^{36} +
                                \numprint{90407322159}\,z^{37} +$\\ & & $
                                \numprint{28792214486}\,z^{38} +
                                \numprint{8164773470}\,z^{39} +
                                \numprint{2047811969}\,z^{40} +
                                \numprint{450329306}\,z^{41} +
                                \numprint{85854230}\,z^{42} +$\\ & & $
                                \numprint{13981660}\,z^{43} +
                                \numprint{1906835}\,z^{44} +
                                \numprint{211876}\,z^{45} +
                                \numprint{18424}\,z^{46} +
                                \numprint{1176}\,z^{47} +
                                \numprint{49}\,z^{48} +
                                z^{49}
    $\\[0.8ex]
    $G_{\overline{8}\times\overline{8}}(z)$ & $=$  & $
                                \numprint{129224}\,z^{16} +
                                \numprint{14681344}\,z^{17} +
                                \numprint{651801600}\,z^{18} +
                                \numprint{15758203520}\,z^{19} +
                                \numprint{240372029072}\,z^{20} +$\\ & & $
                                \numprint{2528654078528}\,z^{21} +
                                \numprint{19500205324032}\,z^{22} +
                                \numprint{115290942264448}\,z^{23} +
                                \numprint{540832229850464}\,z^{24} +$\\ & & $
                                \numprint{2068173372971840}\,z^{25} +
                                \numprint{6588920903240288}\,z^{26} +
                                \numprint{17801592852676672}\,z^{27} +$\\ & & $
                                \numprint{41390172398524272}\,z^{28} +
                                \numprint{83839998055557568}\,z^{29} +
                                \numprint{149484557713246144}\,z^{30} +$\\ & & $
                                \numprint{236656119110649024}\,z^{31} +
                                \numprint{335142837708961654}\,z^{32} +
                                \numprint{427236939021347072}\,z^{33} +$\\ & & $
                                \numprint{492905450386702720}\,z^{34} +
                                \numprint{517004156810313664}\,z^{35} +
                                \numprint{494919960091734336}\,z^{36} +$\\ & & $
                                \numprint{433802866482847616}\,z^{37} +
                                \numprint{349085443267295680}\,z^{38} +
                                \numprint{258463881482739136}\,z^{39} +$\\ & & $
                                \numprint{176377167134882296}\,z^{40} +
                                \numprint{111074953233247104}\,z^{41} +
                                \numprint{64609763870627264}\,z^{42} +$\\ & & $
                                \numprint{34728863089747456}\,z^{43} +
                                \numprint{17251322181046784}\,z^{44} +
                                \numprint{7916762958356992}\,z^{45} +$\\ & & $
                                \numprint{3353820958699552}\,z^{46} +
                                \numprint{1310034044881664}\,z^{47} +
                                \numprint{471036957313244}\,z^{48} +$\\ & & $
                                \numprint{155565089543040}\,z^{49} +
                                \numprint{47060663909504}\,z^{50} +
                                \numprint{12995994842880}\,z^{51} +
                                \numprint{3262480436912}\,z^{52} +$\\ & & $
                                \numprint{740719463168}\,z^{53} +
                                \numprint{151153208768}\,z^{54} +
                                \numprint{27511470912}\,z^{55} +
                                \numprint{4424085048}\,z^{56} +
                                \numprint{621106688}\,z^{57} +$\\ & & $
                                \numprint{74970592}\,z^{58} +
                                \numprint{7624448}\,z^{59} +
                                \numprint{635376}\,z^{60} +
                                \numprint{41664}\,z^{61} +
                                \numprint{2016}\,z^{62} +
                                \numprint{64}\,z^{63} +
                                z^{64}
                                $
 \end{tabular}}
  \caption{Domination polynomials of the torus graph $G_{\overline{n}\times\overline{n}}$.}
  \label{tab:nxn-torus}
\end{table}

\begin{table}
  \centering
  {\small
  \setlength{\tabcolsep}{1pt}
  \begin{tabular}{rcl}
    $K_{1\times 1}(z)$ & $=$ & $z$ \\[0.8ex]
    $K_{2\times 2}(z)$ & $=$  & $
                                \numprint{4}\,z^{1} +
                                \numprint{6}\,z^{2} +
                                \numprint{4}\,z^{3} +
                                z^{4}
    $\\[0.8ex]
    $K_{3\times 3}(z)$ & $=$  & $
                                z^{1} +
                                \numprint{10}\,z^{2} +
                                \numprint{48}\,z^{3} +
                                \numprint{106}\,z^{4} +
                                \numprint{122}\,z^{5} +
                                \numprint{84}\,z^{6} +
                                \numprint{36}\,z^{7} +
                                \numprint{9}\,z^{8} +
                                z^{9}
    $\\[0.8ex]
    $K_{4\times 4}(z)$ & $=$  & $
                                \numprint{256}\,z^{4} +
                                \numprint{1536}\,z^{5} +
                                \numprint{4480}\,z^{6} +
                                \numprint{8320}\,z^{7} +
                                \numprint{10896}\,z^{8} +
                                \numprint{10560}\,z^{9} +
                                \numprint{7744}\,z^{10} +
                                \numprint{4320}\,z^{11} +
                                \numprint{1816}\,z^{12} +$\\ & & $
                                \numprint{560}\,z^{13} +
                                \numprint{120}\,z^{14} +
                                \numprint{16}\,z^{15} +
                                z^{16}
    $\\[0.8ex]
    $K_{5\times 5}(z)$ & $=$  & $
                                \numprint{79}\,z^{4} +
                                \numprint{1593}\,z^{5} +
                                \numprint{14672}\,z^{6} +
                                \numprint{81524}\,z^{7} +
                                \numprint{307244}\,z^{8} +
                                \numprint{842506}\,z^{9} +
                                \numprint{1764068}\,z^{10} +
                                \numprint{2918828}\,z^{11} +$\\ & & $
                                \numprint{3909834}\,z^{12} +
                                \numprint{4311034}\,z^{13} +
                                \numprint{3955232}\,z^{14} +
                                \numprint{3038092}\,z^{15} +
                                \numprint{1957940}\,z^{16} +
                                \numprint{1056965}\,z^{17} +$\\ & & $
                                \numprint{475304}\,z^{18} +
                                \numprint{176256}\,z^{19} +
                                \numprint{53046}\,z^{20} +
                                \numprint{12646}\,z^{21} +$\\ & & $
                                \numprint{2300}\,z^{22} +
                                \numprint{300}\,z^{23} +
                                \numprint{25}\,z^{24} +
                                z^{25}
    $\\[0.8ex]
    $K_{6\times 6}(z)$ & $=$  & $
                                z^{4} +
                                \numprint{56}\,z^{5} +
                                \numprint{1652}\,z^{6} +
                                \numprint{31664}\,z^{7} +
                                \numprint{404770}\,z^{8} +
                                \numprint{3416472}\,z^{9} +
                                \numprint{19840300}\,z^{10} +
                                \numprint{84209540}\,z^{11} +$\\ & & $
                                \numprint{275031868}\,z^{12} +
                                \numprint{718655796}\,z^{13} +
                                \numprint{1546177306}\,z^{14} +
                                \numprint{2797874908}\,z^{15} +
                                \numprint{4326011372}\,z^{16} +$\\ & & $
                                \numprint{5782863816}\,z^{17} +
                                \numprint{6741695574}\,z^{18} +
                                \numprint{6897654436}\,z^{19} +
                                \numprint{6220635186}\,z^{20} +
                                \numprint{4958580672}\,z^{21} +$\\ & & $
                                \numprint{3498131846}\,z^{22} +
                                \numprint{2184049652}\,z^{23} +
                                \numprint{1205216450}\,z^{24} +
                                \numprint{586259808}\,z^{25} +
                                \numprint{250349560}\,z^{26} +$\\ & & $
                                \numprint{93305796}\,z^{27} +
                                \numprint{30113038}\,z^{28} +
                                \numprint{8327600}\,z^{29} +
                                \numprint{1945800}\,z^{30} +
                                \numprint{376864}\,z^{31} +
                                \numprint{58901}\,z^{32} +$\\ & & $
                                \numprint{7140}\,z^{33} +
                                \numprint{630}\,z^{34} +
                                \numprint{36}\,z^{35} +
                                z^{36}
    $\\[0.8ex]
    $K_{7\times 7}(z)$ & $=$  & $
                                \numprint{243856}\,z^{9} +
                                \numprint{7483274}\,z^{10} +
                                \numprint{108525780}\,z^{11} +
                                \numprint{995661210}\,z^{12} +
                                \numprint{6526376452}\,z^{13} +
                                \numprint{32723647242}\,z^{14} +$\\ & & $
                                \numprint{131188032404}\,z^{15} +
                                \numprint{433817785292}\,z^{16} +
                                \numprint{1211009331050}\,z^{17} +
                                \numprint{2904839371392}\,z^{18} +$\\ & & $
                                \numprint{6071176663246}\,z^{19} +
                                \numprint{11178937768294}\,z^{20} +
                                \numprint{18295752974580}\,z^{21} +$\\ & & $
                                \numprint{26804759801972}\,z^{22} +
                                \numprint{35356180710524}\,z^{23} +
                                \numprint{42178267079370}\,z^{24} +$\\ & & $
                                \numprint{45670952317403}\,z^{25} +
                                \numprint{45011034604106}\,z^{26} +
                                \numprint{40458849573846}\,z^{27} +$\\ & & $
                                \numprint{33215036685152}\,z^{28} +
                                \numprint{24925366211032}\,z^{29} +
                                \numprint{17102403546926}\,z^{30} +$\\ & & $
                                \numprint{10726989678404}\,z^{31} +
                                \numprint{6145751104023}\,z^{32} +
                                \numprint{3212103217512}\,z^{33} +$\\ & & $
                                \numprint{1528690222560}\,z^{34} +
                                \numprint{660843701416}\,z^{35} +
                                \numprint{258681402216}\,z^{36} +
                                \numprint{91330527514}\,z^{37} +$\\ & & $
                                \numprint{28943075360}\,z^{38} +
                                \numprint{8183779088}\,z^{39} +
                                \numprint{2049421399}\,z^{40} +
                                \numprint{450371272}\,z^{41} +
                                \numprint{85843308}\,z^{42} +$\\ & & $
                                \numprint{13979844}\,z^{43} +
                                \numprint{1906704}\,z^{44} +
                                \numprint{211872}\,z^{45} +
                                \numprint{18424}\,z^{46} +
                                \numprint{1176}\,z^{47} +
                                \numprint{49}\,z^{48} +
                                z^{49}
    $\\[0.8ex]
    $K_{8\times 8}(z)$ & $=$  & $
                                \numprint{3600}\,z^{9} +
                                \numprint{260234}\,z^{10} +
                                \numprint{9161844}\,z^{11} +
                                \numprint{205624178}\,z^{12} +
                                \numprint{3259026956}\,z^{13} +
                                \numprint{38509091104}\,z^{14} +$\\ & & $
                                \numprint{351743132940}\,z^{15} +
                                \numprint{2555393428502}\,z^{16} +
                                \numprint{15128696395436}\,z^{17} +
                                \numprint{74541297707306}\,z^{18} +$\\ & & $
                                \numprint{311267686259112}\,z^{19} +
                                \numprint{1118844024839124}\,z^{20} +
                                \numprint{3507981273108664}\,z^{21} +$\\ & & $
                                \numprint{9702498525018636}\,z^{22} +
                                \numprint{23899882018866672}\,z^{23} +
                                \numprint{52858603217834524}\,z^{24} +$\\ & & $
                                \numprint{105690774510597180}\,z^{25} +
                                \numprint{192179344747568048}\,z^{26} +
                                \numprint{319368084410733612}\,z^{27} +$\\ & & $
                                \numprint{487117660190269044}\,z^{28} +
                                \numprint{684379499046113744}\,z^{29} +
                                \numprint{888386977466277426}\,z^{30} +$\\ & & $
                                \numprint{1068217222601672912}\,z^{31} +
                                \numprint{1192321377072934280}\,z^{32} +
                                \numprint{1237548909927735548}\,z^{33} +$\\ & & $
                                \numprint{1196127084749768650}\,z^{34} +
                                \numprint{1077740592175963352}\,z^{35} +
                                \numprint{905994491238380692}\,z^{36} +$\\ & & $
                                \numprint{710965651477267076}\,z^{37} +
                                \numprint{520969168389552836}\,z^{38} +
                                \numprint{356483920242132856}\,z^{39} +$\\ & & $
                                \numprint{227748014955114180}\,z^{40} +
                                \numprint{135792828381540616}\,z^{41} +
                                \numprint{75513548059989048}\,z^{42} +$\\ & & $
                                \numprint{39130117374538132}\,z^{43} +
                                \numprint{18872828052876618}\,z^{44} +
                                \numprint{8460284139138604}\,z^{45} +$\\ & & $
                                \numprint{3518912510054954}\,z^{46} +
                                \numprint{1355245912038020}\,z^{47} +
                                \numprint{482129585758940}\,z^{48} +$\\ & & $
                                \numprint{157983537865980}\,z^{49} +
                                \numprint{47524258972966}\,z^{50} +
                                \numprint{13073010514020}\,z^{51} +$\\ & & $
                                \numprint{3273341812692}\,z^{52} +
                                \numprint{741978339844}\,z^{53} +
                                \numprint{151266210264}\,z^{54} +
                                \numprint{27518246208}\,z^{55} +$\\ & & $
                                \numprint{4424188406}\,z^{56} +
                                \numprint{621078384}\,z^{57} +
                                \numprint{74967272}\,z^{58} +
                                \numprint{7624272}\,z^{59} +
                                \numprint{635372}\,z^{60} +
                                \numprint{41664}\,z^{61} +$\\ & & $
                                \numprint{2016}\,z^{62} +
                                \numprint{64}\,z^{63} +
                                z^{64}
           $
 \end{tabular}}
  \caption{Domination polynomials of the king graph $K_{n \times n}$.}
  \label{tab:nxn-king}
\end{table}

\bibliographystyle{jis}
\bibliography{math,mertens,cs}


\bigskip
\hrule
\bigskip

\noindent 2020 {\it Mathematics Subject Classification}:
Primary
  05C69   
  
Secondary
  05A15;  
  05C30,  
  11B83  
 
\noindent \emph{Keywords:}
Domination polynomial, grid graph, cylinder graph, torus graph, king
graph, dominating set, transfer matrix, algorithm

\bigskip
\hrule
\bigskip

\noindent (Concerned with sequences
\seqnum{A001333},
\seqnum{A078057},
\seqnum{A030270},
\seqnum{A075561},
\seqnum{A094087},
\seqnum{A104519},
\seqnum{A124696},
\seqnum{A133515},
\seqnum{A133791},
\seqnum{A208716},
\seqnum{A218354},
\seqnum{A218663},
\seqnum{A286514},
\seqnum{A303334},
\seqnum{A347554},
\seqnum{A347557},
\seqnum{A347632},
\seqnum{A350815},
\seqnum{A350820}
)

\end{document}